\documentclass[12pt]{amsart}
\usepackage[a4paper]{geometry}
\usepackage{amsthm,hyperref,lmodern,mathrsfs}

\usepackage[latin1]{inputenc}
\usepackage[T1]{fontenc} 
\usepackage{graphicx}
\usepackage[frenchb,english]{babel}
\usepackage{calc}
\usepackage{amsmath,amsthm}

\usepackage{url}
\usepackage{times, amssymb, amscd, mathrsfs, graphicx, color}  
\usepackage{enumerate}

\newtheorem{theo}{Theorem}[section]
\newtheorem{prop}[theo]{Proposition}
\newtheorem{defi}[theo]{Definition}

\newtheorem{lem}[theo]{Lemma}

\newtheorem{Rq}[theo]{Remark}

\title[]{Estimating the covariance of Random Matrices}
\author{Pierre Youssef}
\begin{document}

\maketitle

\begin{abstract}
We extend to the matrix setting a recent result of 
Srivastava-Vershynin \cite{srivastava-vershynin} 
about estimating the covariance matrix of a random 
vector. The result can be interpreted as a quantified version 
of the law of large numbers for  positive 
semi-definite matrices which verify some regularity 
assumption. Beside giving examples, we discuss the notion of log-concave 
matrices and give estimates on the smallest and largest 
eigenvalues of a sum of such matrices.

\end{abstract}

\section{Introduction}

In recent years, interest in matrix valued random variables 
gained momentum. Many of the results dealing with real 
random variables and random vectors were extended to 
cover random matrices. Concentration inequalities like 
Bernstein, Hoeffding and others were obtained in the 
non-commutative setting ( \cite{MR1966716},\cite{MR2946459}, \cite{jordan-tropp}).
 The methods used were mostly combination of methods from the 
real/vector case and some matrix inequalities like the 
Golden-Thompson inequality (see \cite{MR1477662}).
\vskip 0.3cm

    Estimating the 
  covariance matrix of a random vector has gained a lot of
   interest recently. Given a random vector $X$ in $\mathbb{R}^n$,
    the question is to estimate $\Sigma =\mathbb{E}XX^t $. A 
natural way to do this is to take $X_1,..,X_N$ independent copies 
of $X$ and try to approximate $\Sigma$ with the sample covariance
 matrix $\Sigma_N = \frac{1}{N} \sum_i X_iX_i^t$.
The challenging problem is to find the minimal number of samples
needed to estimate $\Sigma$. It is known using a result of 
Rudelson (see \cite{MR1694526}) that for general distributions
 supported on the sphere of radius $\sqrt{n}$, it suffices to take $cn\log(n)$ samples. 
But for many distributions, a number proportional to $n$ is sufficient.
Using standard arguments, one can verify this for gaussian vectors.
It was conjectured by Kannan- Lovasz- Simonovits \cite{MR1608200} 
that the same result holds for log-concave distributions. 
This problem was solved 
by Adamczak et al (\cite{MR2601042}, \cite{MR2769907}). Recently, Srivatava-Vershynin 
proved in \cite{srivastava-vershynin} covariance estimate with 
a number of samples proportional to $n$, for a larger class
of distributions covering the log-concave case. The method 
used was different from previous work on this field and the 
main idea was to randomize the sparsification theorem of 
Batson-Spielman-Srivastava \cite{batson-spielman-srivastava}.

\vskip 0.3cm

Our aim in this paper is to adapt the work of Srivastava-Vershynin 
to the matrix setting replacing the vector $X$ in the problem 
of the covariance matrix by an $n\times m$ random matrix $A$ 
and try to estimate $\mathbb{E}AA^t$ by the same techniques. 
This will be possible since in the deterministic setting, the 
sparsification theorem of Batson-Spielman-Srivastava \cite{batson-spielman-srivastava} 
has been extended to a matrix setting by De Carli Silva-Harvey-Sato \cite{silva-harvey-sato} 
who precisely proved the following:

\begin{theo} 
 Let $B_1,\ldots,B_m$ be positive semi-definite matrices of size $n \times n$
and arbitrary rank.
Set $B := \sum_i B_i$.
For any $\varepsilon \in (0,1)$, there is a deterministic algorithm to construct a vector $y \in \mathbb{R}^m$ with
$O(n / \varepsilon^2)$ nonzero entries such that $y \geq 0$ and
$$
B ~\preceq~ \sum_i y_i B_i ~\preceq~ (1+\varepsilon) B.
$$   
\end{theo}

For an $n\times n$ matrix $A$, denote by $\Vert A\Vert$ the operator norm 
of $A$ seen as an operator on $l_2^n$. The main idea  is to randomize the previous result 
using the techniques of Srivastava-Vershynin \cite{srivastava-vershynin}. 
Our problem can be formulated as follows:\\

Take $B$ a positive semi-definite random matrix 
of size $n\times n$.  How many independent copies of $B$ 
are needed to approximate $\mathbb{E}B$ i.e taking $B_1,
..,B_N$ independent copies of $B$, what is the minimal number
 of samples needed to make $\left\Vert \frac{1}{N} \sum_i B_i -
 \mathbb{E}B\right\Vert$ very small.
 
\vskip 0.2cm
 
 One can view this as a matrix analogue to the covariance estimate
 of a random vector by taking for $B$ the matrix $AA^t$ where
  $A$ is an $n\times m$ random matrix. Moreover, this problem implies 
  an averaging approximation of covariance matrices of many random vectors. 
  Indeed, let $X_1,..,X_m$ be random 
  vectors in $\mathbb{R}^n$ and take $A'$ the $n\times m$ matrix which has 
  $X_1,..,X_m$ as columns. Denote $B= \frac{1}{m}A'A'^t$; it is clear that 
  $B= \frac{1}{m}\sum_{j\leqslant m} X_jX_j^t$. Therefore, when approximating 
  $\mathbb{E}B$ we are approximating the average of the covariance matrices of 
  the random vectors $(X_j)_{j\leqslant m}$.

  With some regularity, we will be able to take a number of independent copies 
  proportional to the dimension $n$. 
However, in the general case this is no longer true.
 In fact, take $B$ uniformly distributed on $\left\{n e_i
 e_i^t\right\}_{i\leqslant n}$ where $e_j$ denotes the 
 canonical basis of $\mathbb{R}^n$. It is easy to verify that
  $\mathbb{E} B= I_n$ and when taking $B_1,..,B_N$ independent copies of $B$ 
the matrix $\frac{1}{N}\sum_i B_i $ is diagonal and its diagonal coefficients are distributed
   as $$\frac{n}{N}(p_1,..,p_n),$$
   where $p_i$ denotes the number of times $e_ie_i^t$ is chosen. This
    problem is well- studied and it is known (see \cite{MR0471016}) 
    that we must take $N \geqslant  cn\log(n)$. This example is 
    essentially due to Aubrun  \cite{aubrun}. More generally, 
if $B$ is a positive semi-definite matrix such that $\mathbb{E}B=I_n$
 and ${\rm Tr}(B)\leqslant n$ almost surely, then 
by Rudelson's inequality in the non-commutative setting (see \cite{MR2653725}) 
it is sufficient to take $cn\log(n)$ samples.
    
\vskip 0.2cm

 The method will work properly for a class of matrices 
 satisfying a matrix strong regularity assumption which we denote 
 by $(MSR)$ and can be viewed as an analog to the property 
 $(SR)$ defined in  \cite{srivastava-vershynin}.

 \begin{defi}\label{def-MSR}[\text{Property $(MSR)$}]\\
  Let $B$ be an $n\times n$ positive semi-definite random 
 matrix such that $\mathbb{E}B =I_n$. We will say that 
 $B$ satisfies $(MSR)$ if for some $c, \eta >0$ we have :
 $$
 \mathbb{P} (\left\Vert PBP\right\Vert \geqslant t) \leqslant \frac{c}{t^{1+\eta}} 
 \quad \forall t\geqslant c\cdot rank(P) \text{ and $\forall P$ 
 orthogonal projection of $\mathbb{R}^n$}.
 $$
 \end{defi}

In the rest of the paper, $c$ will always denote the parameter appearing in this definition while 
$C$ will be a universal constant which may change from line to line. Also, $c(\eta)$ will denote 
a constant depending on $c$ and $\eta$ which may also change from line to line.

The main result of this paper is the following:

\begin{theo}\label{covariance-estimate}
Let $B$ be an $n\times n$  positive semi-definite 
random matrix verifying $\mathbb{E} B= I_n$ and $(MSR)$ for some $\eta >0$.
Then for every $\varepsilon \in (0,1)$, taking 
$N=C_1\footnote{$C_1 = (64c)^{1+\frac{2}{\eta}}
 (1+\frac{1}{\eta})^{\frac{2}{\eta}} \vee 64 (4c)^{\frac{1}{\eta}} 
 (32+\frac{32}{\eta})^{1+\frac{3}{\eta}} \vee 256 (2c)^{\frac{3}{2} +\frac{2}{\eta}}
  (16 +\frac{16}{\eta})^{\frac{4}{\eta}} $ } 
  \frac{n}{\varepsilon^{2+\frac{2}{\eta}}}$ we have
$$ \displaystyle \mathbb{E} \left\Vert \frac{1}{N} \sum_{i=1}^N B_i -I_n\right\Vert 
\leqslant \varepsilon \quad \text{ where $B_1,..,B_N$ are independent copies of $B$}.$$
\end{theo}

If $X$ is an isotropic random vector of 
  $\mathbb{R}^n$, put $B=XX^t$ then 
  $\left \Vert PBP\right\Vert=\Vert PX\Vert_2^2$. Therefore if $X$ verifies the 
  property $(SR)$ appearing in \cite{srivastava-vershynin}, 
  then $B$ verifies property $(MSR)$. So applying Theorem 
  \ref{covariance-estimate} to $B=X X^t$, we recover
   the covariance estimation as stated in \cite{srivastava-vershynin}.

\vskip 0.2cm

In order to apply our result, beside some examples, we investigate the 
notion of log-concave matrices in relation with the definition of log-concave vectors. 
Moreover remarking some strong concentration 
inequalities satisfied by these matrices we are able, using the ideas developed in 
the proof of the main theorem, to have some results with 
high probability rather than only in expectation as is the case in the main result. This will be 
discussed in the last section of the paper.
\vskip 0.2cm

The paper is organized as follows: in section 2, we discuss Property $(MSR)$ and give some examples, 
in section 3 we show how to prove 
Theorem~\ref{covariance-estimate} using two other results (Theorem~\ref{lower-bound}, 
Theorem~\ref{upper-bound}) which we prove 
respectively in sections 4 and 5 using again two other results (Theorem~\ref{iteration-lower-bound}, 
Theorem~\ref{iteration-upper-bound}) whose proofs are given respectively in sections 6 and 7.  In section 8, we 
discuss the notion of log-concave matrices and prove some related results.
  
\vskip 0.5cm

\section{Property $(MSR)$ and examples}

\vskip 0.5cm

A random vector $X$ in $\mathbb{R}^l$ is called 
isotropic if its covariance matrix is the identity i.e 
$\mathbb{E} XX^t= Id$. In \cite{srivastava-vershynin}, an isotropic 
random vector $X$ in $\mathbb{R}^l$ was 
said to satisfy $(SR)$ if for some $c, \eta>0$,
$$
 \mathbb{P} \left(\left\Vert PX\right\Vert_2^2 
\geqslant t\right) \leqslant \frac{c}{t^{1+\eta}}, 
\quad \forall t\geqslant c\cdot rank(P) \text{ and 
$\forall P$ orthogonal projection of $\mathbb{R}^{l}$}.  
$$

Since $\Vert PXX^tP\Vert =\Vert PX\Vert_2^2$, then clearly 
$B=XX^t$ satisfies $(MSR)$ if and only if $X$ satisfies $(SR)$. 
Therefore if $X$ verifies the property $(SR)$, applying 
Theorem~\ref{covariance-estimate} to $B=X X^t$, we recover
   the covariance estimate as stated in \cite{srivastava-vershynin}.

Let us note that $(MSR)$ implies moment assumptions 
on the quadratic forms $\langle Bx,x\rangle$. To see this, 
first note that if $x\in\mathbb{S}^{n-1}$ then $\langle Bx,x\rangle = \Vert P_xBP_x\Vert$, 
where $P_x$ is the orthogonal projection on ${\rm span}(x)$. Now, by 
integration of tails we have for $1<q< 1+\eta$, $$ \mathbb{E} \langle Bx,x\rangle^{q} \leqslant C(c,q,\eta).$$

Moreover, property $(MSR)$ implies regularity assumption on the 
eigenvalues of the matrix $B$. Indeed, for any orthogonal projection of 
rank $k$ one can write
$$
\Vert PBP\Vert \geqslant 
\min_{\underset{{\rm dim}F=k}{F \subset\mathbb{R}^n}}
\max_{x\in F} \langle Bx,x\rangle
=\lambda_{n-k+1}(B),
$$
where the last equality is given by the Courant-Fisher minimax formula (see \cite{lama-book}). 
Therefore, property $(MSR)$ implies the following: for some $c, \eta>0$, 
$$
 \mathbb{P} \left(\lambda_{n-k+1}(B)
\geqslant t\right) \leqslant \frac{c}{t^{1+\eta}}, 
\quad \forall t\geqslant c\cdot k \text{ and } \forall k\leqslant n.  
$$

We may now discuss some examples for applications of the main result.  
Let us first replace $(MSR)$ with a stronger, but easier to manipulate, property which we denote by 
$(MSR*)$. If $B$ is an $n\times n$ positive semidefinite random matrix such that $\mathbb{E} B=I_n$, we 
will say that $B$ satisfies $(MSR*)$ if for some $c,\eta> 0$:
$$
\mathbb{P} ({\rm Tr}(PB)\geqslant t) \leqslant \frac{c}{t^{1+\eta}} 
 \quad \forall t\geqslant c\cdot rank(P) \text{ and $\forall P$ 
 orthogonal projection of $\mathbb{R}^n$}.  
 $$

Note that since $\Vert PBP\Vert \leqslant {\rm Tr}\left( PBP\right)={\rm Tr}\left( PB\right)$, 
then $(MSR*)$ is clearly stronger than $(MSR)$.\\

\vskip 0.5cm

\subsection{$(2+\varepsilon)$-moments for the spectrum }\hfill\\

As we mentioned before, $(MSR)$, one can see that it implies 
regularity assumptions on the eigenvalues of $B$. Putting 
some independence in the spectral decomposition of $B$, 
we will only need to use the regularity of the eigenvalues. 
To be more precise, we have the following:

\begin{prop}
Let $B= UDU^t$ be the spectral decomposition of an $n\times n$ symmetric positive semidefinite random matrix, where $U$ 
an orthogonal matrix and $D$ is a diagonal matrix whose entries are denoted by $(\alpha_j)_{j\leqslant n}$. 
Suppose that $U$ and $D$ are independent and that 
$(\alpha_j)_{j\leqslant n}$ are independent and satisfy the following:
$$
\forall i\leqslant n, \ \mathbb{E} \alpha_i =1 \quad \text{ and } \quad \left(\mathbb{E} \alpha_i^p\right)^{\frac{1}{p}} \leqslant C,
$$
for some $p> 2$. Then $B$ satisfies $(MSR*)$ with\quad  $\eta =\frac{p}{2}-1$.
\end{prop}

\begin{proof}[Proof]
First note that since $U$ and $D$ are independent and $\mathbb{E} \alpha_i =1$, then $\mathbb{E} B= I_n$. 

Now, $(MSR*)$ is a rotationally invariant property. Therefore we can assume without loss of generality that $U = I_n$ and 
thus that $B=D$. 
Let $k>0$ and $P$ be an orthogonal projection of rank $k$ on $\mathbb{R}^n$ and denote by $(p_{ij})_{i,j\leqslant n}$ the entries of $P$. 
Note that ${\rm Tr}\left(PB\right) = \sum_{i\leqslant n} p_{ii} \alpha_i$, 
and now using Markov's inequality we have for $t> k$,
\begin{align*}
\mathbb{P}\left\{ {\rm Tr}\left( PB\right)\geqslant t\right\}&= \mathbb{P}\left\{  \sum_{i\leqslant n} p_{ii} (\alpha_i -1)
 \geqslant t-k\right\} \\
 &\leqslant \mathbb{P}\left\{  \left\vert\sum_{i\leqslant n} p_{ii} (\alpha_i -1)\right\vert
 \geqslant t-k\right\} \\
 &\leqslant \frac{1}{(t-k)^p} \mathbb{E}\left\vert \sum_{i\leqslant n} p_{ii} (\alpha_i -1)\right\vert^p .
\end{align*}
Using Rosenthal's inequality (see \cite{MR0271721}) we get
$$
\mathbb{E}\left\vert \sum_{i\leqslant n} p_{ii} (\alpha_i -1)\right\vert^p
\leqslant C(p) \max\left\{ \sum_{i\leqslant n}p_{ii}^p\mathbb{E}\vert \alpha_i -1\vert^p , 
\left( \sum_{i\leqslant n}p_{ii}^2\mathbb{E}\vert \alpha_i -1\vert^2\right)^{\frac{p}{2}}\right\} .
$$
Taking in account that $p_{ii}\leqslant 1$, which implies that for any $l\geqslant 1$, $\sum_{i} p_{ii}^l\leqslant k$,
 we deduce that
$$
\mathbb{E}\left\vert \sum_{i\leqslant n} p_{ii} (\alpha_i -1)\right\vert^p
\leqslant C(p) k^{\frac{p}{2}}  .
$$
Instead of Rosenthal's inequality, we could have used a symmetrization argument alongside Khintchine's inequality 
to get the estimate above. \\
One can easily conclude that $B$ satisfies $(MSR*)$ with $\eta= \frac{p}{2} -1$.
\end{proof}

Applying Theorem~\ref{covariance-estimate}, we can deduce the following proposition:

\begin{prop}
Let $B= UDU^t$ be the spectral decomposition of an $n\times n$ symmetric positive semidefinite random matrix, where $U$ 
an orthogonal matrix and $D$ is a diagonal matrix whose entries are denoted by $(\alpha_j)_{j\leqslant n}$. 
Suppose that $U$ and $D$ are independent and that 
$(\alpha_j)_{j\leqslant n}$ are independent and satisfy the following:
$$
\forall i\leqslant n, \ \mathbb{E} \alpha_i =1 \quad \text{ and } \quad \left(\mathbb{E} \alpha_i^p\right)^{\frac{1}{p}} \leqslant C,
$$
for some $p> 2$. Let $\varepsilon \in (0,1)$, then taking $N= C(p)\frac{n}{\varepsilon^{\frac{2p}{p-2}}}$  we have
$$ \displaystyle \mathbb{E} \left\Vert \frac{1}{N} \sum_{i=1}^N B_i -I_n\right\Vert 
\leqslant \varepsilon,$$ where $B_1,..,B_N$ are independent copies of $B$.\\

\end{prop}

\subsection{From $(SR)$ to $(MSR)$}\hfill\\

We will show how to jump from property $(SR)$ dealing 
with vectors to the property $(MSR*)$ dealing with matrices.

\begin{prop}\label{vector-to-matrix}
Let $A$ be an $n\times m$ random matrix and 
denote by $(C_i)_{i\leqslant m}$ its columns. 
Suppose that $A'^t=\sqrt{m}(C_1^t,..,C_m^t)$ is an isotropic 
random vector in $\mathbb{R}^{nm}$ which satisfies 
property $(SR)$. 
Then $B =AA^t$ verifies 
$\mathbb{E} B=I_n$ and Property $(MSR*)$.
\end{prop}

\begin{proof}[Proof]
For $l\leqslant nm$, one can write $l= (j-1)n+ i$ 
with $1\leqslant i\leqslant n,\ 1\leqslant j\leqslant m$, 
so that the coordinates of $A'$ are given by $a'_l= \sqrt{m}a_{i,j}$, 
and since $A'$ is isotropic we get $\mathbb{E} a_{i,j} a_{r,s}=
\frac{1}{m} \mathbb{\delta}_{(i,j),(r,s)}$.\\
The terms of $B$ are given by $\displaystyle b_{i,j}= 
\sum_{s=1}^m a_{i,s}a_{j,s}$. We deduce that
 $\mathbb{E} b_{i,j}= \delta_{i,j}$ and therefore $\mathbb{E}B= I_n$.\\
Let $P$ be an orthogonal projection of $\mathbb{R}^n$ 
and put $P'=I_m\otimes P$ i.e. $P'$ is an $nm\times nm$ matrix of the form 
$$
P'=\left(\begin{array}{cccc}
P & 0&\ldots &0\\
0&P& \ddots&\vdots\\
\vdots &\ddots & \ddots&\vdots\\
0 &\ldots &\ldots &P
\end{array}\right)
$$

Clearly we have $\left\Vert P'A'\right\Vert_2^2
 = m {\rm Tr}(PB)$ and $rank(P')=m\cdot rank(P)$.\\
 Let $t\geqslant c\cdot rank(P)$ then $mt\geqslant c\cdot rank(P')$ and 
 by property $(SR)$ we have
$$
 \mathbb{P}\left( \left\Vert P'A'\right\Vert_2^2 \geqslant mt\right)\leqslant \frac{c}{(mt)^{1+\eta}} .
$$
This means that 
$$
 \mathbb{P} \left({\rm Tr}(PB)\geqslant t\right) \leqslant \frac{c}{(mt)^{1+\eta}} ,
$$
and therefore $B$ satisfies $(MSR*)$.
\end{proof}

\hfill 
\section{Proof of Theorem~\ref{covariance-estimate} }\hfill

Let us first introduce some notations which will be used in the rest of the paper. 
The set of $n\times n$ symmetric matrices is denoted by $\mathbb{S}_n$. For 
$X\in \mathbb{S}_n$, the notation $\lambda (X)$ always refers to the eigenvalues of $X$. 
For $X, Y \in \mathbb{S}_n$, $X\succeq Y$ means that $X-Y$ is positive semidefinite. 
The vector space $\mathbb{S}_n$ can be endowed with the trace inner product $\langle \cdot,\cdot\rangle$ 
defined by $\langle X,Y\rangle := {\rm Tr}(XY).$
\ \\

Let us now introduce a regularity assumption on the moments 
which we denote by $(MWR)$: 
 $$
\exists p>1 \text{ \ such that \ } \mathbb{E}\left<Bx,x\right>^{p} \leqslant 
 C_p
\quad \forall x\in S^{n-1}.
$$
Note that 
by a simple integration of tails, 
$(MSR)$ (with $P$ a rank one projection) implies $(MWR)$ with $p<1+\eta$.

The proof of Theorem~\ref{covariance-estimate} is based on two theorems 
dealing with the smallest and largest eigenvalues of $\displaystyle \frac{1}{N}\sum_{i=1}^N B_i$.

\begin{theo}\label{lower-bound}
Let $B_i$ be $n\times n$ independent  
positive semidefinite random matrices verifying 
$\mathbb{E} B_i=I_n$ and $(MWR)$ .\\
 Let $\varepsilon \in (0,1)$, then for $$N\geqslant 
 16\left(16 C_p\right)^{\frac{1}{p-1}}
 \frac{n}{\varepsilon^{\frac{2p-1}{p-1}}},$$  we get
$$\mathbb{E}\displaystyle \lambda_{min}\left( 
\frac{1}{N}\sum_{i=1}^N B_i\right) \geqslant 1-\varepsilon .$$
\end{theo}

\begin{Rq}
The proof yields a more general 
estimate; precisely if $h=\frac{n}{N}$ then 
$$
\mathbb{E}\displaystyle \lambda_{min}\left( 
\frac{1}{N}\sum_{i=1}^N B_i\right) \geqslant 
1-C(p)\max\left\{h^{\frac{p-1}{2p-1}},h\right\} .
$$
\end{Rq}

\begin{theo}\label{upper-bound}
Let $B_i$ be $n\times n$ independent  
positive semidefinite random matrices verifying 
$\mathbb{E} B_i=I_n$ and $(MSR)$.\\
For $\varepsilon \in (0,1)$ and $N\geqslant 
C_{2}\footnote{$C_{2}= 16 c^{\frac{1}{\eta}} 
(32 +\frac{32}{\eta})^{1+\frac{3}{\eta}}  \vee 16 \sqrt{2} (4c)^{\frac{3}{2} + \frac{2}{\eta}} 
(8+\frac{8}{\eta})^{\frac{4}{\eta}} $}   
\frac{n}{\varepsilon^{2+\frac{2}{\eta}}}$ we have

$$\mathbb{E}\displaystyle \lambda_{\max}\left( \frac{1}{N}\sum_{i=1}^N B_i\right) 
\leqslant 1+\varepsilon .$$

\end{theo}

\begin{Rq}
The proof yields a more general estimate; 
Precisely if $h=\frac{n}{N}$ then 
$$
\mathbb{E}\displaystyle \lambda_{max}\left( 
\frac{1}{N}\sum_{i=1}^N B_i\right) \leqslant 
1+c(\eta)\max\left\{h^{\frac{\eta}{2+2\eta}},h\right\} . 
$$
Combining this with the previous remark, 
for any $B_1,...,B_N$ $n\times n$ independent  
positive semidefinite random matrices verifying 
$\mathbb{E} B_i=I_n$ and $(MSR)$, we have 
$$
1-c(\eta)\max\left\{h^{\frac{\eta}{2+2\eta}},h\right\}\leqslant 
\mathbb{E}\displaystyle \lambda_{min}\left( 
\frac{1}{N}\sum_{i=1}^N B_i\right)\leqslant 
\mathbb{E}\displaystyle \lambda_{max}\left( 
\frac{1}{N}\sum_{i=1}^N B_i\right) \leqslant 
1+c(\eta)\max\left\{h^{\frac{\eta}{2+2\eta}},h\right\}
$$

\end{Rq}

We will give the proof of these two theorems in sections 4 and 5 respectively. 
We need also a simple lemma:

\begin{lem}\label{lem-firststep}
Let  $1<r\leqslant 2$ and $Z_1,...,Z_N$ be independent positive 
random variables 
 with $\mathbb{E} Z_i=1$ and satisfying
 $
\left( \mathbb{E} Z_i^r\right)^{\frac{1}{r}} \leqslant M
 $
Then
  $$\displaystyle\mathbb{E}  \left\vert \frac{1}{N}
  \sum_{i=1}^N Z_i -1\right\vert \leqslant \frac{2M}{N^{\frac{r-1}{r}}} .$$
\end{lem}

\begin{proof}
Let $(\varepsilon_i)_{i\leqslant N}$ independent $\pm 1$ Bernoulli variables. 
By symmetrization and Jensen's inequality we can write
\begin{align*}
\mathbb{E}  \left\vert \frac{1}{N}
  \sum_{i=1}^N Z_i -1\right\vert &\leqslant \frac{2}{N} \mathbb{E}
  \left\vert
  \sum_{i=1}^N \varepsilon_iZ_i \right\vert \leqslant \frac{2}{N} \mathbb{E} \left(\sum_{i=1}^N Z_i^2\right)^{\frac{1}{2}}\\
&\leqslant \frac{2}{N} \mathbb{E} \left(\sum_{i=1}^N Z_i^r\right)^{\frac{1}{r}}\leqslant \frac{2}{N}  \left(\sum_{i=1}^N \mathbb{E}Z_i^r\right)^{\frac{1}{r}}\leqslant \frac{2M}{N^{\frac{r-1}{r}}}
.  \end{align*}
\end{proof}

\begin{proof}[Proof of Theorem~\ref{covariance-estimate}]
Take $N\geqslant c(\eta)\frac{n}{\varepsilon^{2+\frac{2}{\eta}}}$ satisfying 
conditions of Theorem~\ref{lower-bound} (with $p= 1+\frac{\eta}{2}$) and Theorem~\ref{upper-bound}. 
Note that by the triangle inequality 

\begin{align*}
 \left\Vert \displaystyle \frac{1}{N}  \sum_{i=1}^N B_i -I_n
\right\Vert &\leqslant  \left\Vert \frac{1}{N} \displaystyle 
\sum_{i=1}^N B_i -\frac{1}{n}{\rm Tr}\left(\frac{1}{N}\sum_{i=1}^N B_i\right)
I_n \right\Vert + \left\Vert \frac{1}{n}{\rm Tr}\left(\frac{1}{N}
\sum_{i=1}^N B_i\right)I_n - I_n\right\Vert\\ \\
&:= \alpha +\beta .\\
\end{align*}
Observe that
\begin{align*}
 \alpha &= \max\left\vert\lambda  \left(  \frac{1}{N}\sum_{i=1}^N B_i 
 -\frac{1}{n}{\rm Tr}\left(\frac{1}{N}\sum_{i=1}^N B_i\right)I_n \right)\right\vert\\
	    &=\max\left[ \lambda_{\max} \left(  \frac{1}{N}\sum_{i=1}^N B_i
	    \right)-\frac{1}{n}{\rm Tr}\left(\frac{1}{N}\sum_{i=1}^N B_i\right), 
	    \frac{1}{n}{\rm Tr}\left(\frac{1}{N}\sum_{i=1}^N B_i\right)-\lambda_{\min} 
	    \left(  \frac{1}{N}\sum_{i=1}^N B_i\right)\right] .\\	   
\end{align*}

Since the two terms in the max are non-negative, 
then one can bound the max by the sum of the two terms. 
More precisely, we get $\alpha\leqslant \displaystyle 
\lambda_{\max} \left(  \frac{1}{N}\sum_{i=1}^N B_i\right)- 
\lambda_{\min} \left(  \frac{1}{N}\sum_{i=1}^N B_i\right)$ and 
by Theorem~\ref{lower-bound} and Theorem~\ref{upper-bound} 
we deduce that 
$\mathbb{E} \alpha \leqslant 2\varepsilon$.
\ \ 
\vskip 0.2cm

Note that $$ \beta =\left\vert 
\frac{1}{N} \sum_{i=1}^N  \frac{{\rm Tr}(B_i)}{n} -1\right\vert =\left\vert 
\frac{1}{N} \sum_{i=1}^N  Z_i -1\right\vert,$$
where $Z_i= \frac{{\rm Tr}(B_i)}{n}$. Since $B_i$ satisfies $(MWR)$, then 
taking $r=\min (2, 1+\frac{\eta}{2})$ we have 
$$
\forall {i\leqslant N},\quad \left(\mathbb{E}  Z_i^r\right)^{\frac{1}{r}} \leqslant \frac{1}{n}\sum_{j=1}^n \left(\mathbb{E} \left<B_ie_j, e_j\right>^r\right)^{\frac{1}{r}} \leqslant c(\eta).
$$ 

Therefore $Z_i$ satisfy the conditions of Lemma~\ref{lem-firststep} and we deduce that $\mathbb{E}\beta \leqslant \varepsilon$ by the choice of $N$.

As a conclusion
$$\mathbb{E} \left\Vert \displaystyle \frac{1}{N}  \sum_{i=1}^N 
B_i -I_n\right\Vert \leqslant \mathbb{E} \alpha +\mathbb{E} \beta 
\leqslant 3\varepsilon .$$

 \end{proof}

\hfill

\section{Proof of Theorem~\ref{lower-bound}}\hfill

Given $A$ an $n\times n$  positive 
semi-definite matrix such that all eigenvalues of $A$ are greater than a 
lower barrier $l_A=l$ i.e $A\succ l.I_n$, define the 
corresponding potential function to be 
$$
\phi_l(A)= {\rm Tr}(A-l\cdot I_n)^{-1}.
$$
The proof of Theorem~\ref{lower-bound} is based on 
the following result which will be proved in section 6:

\begin{theo}\label{iteration-lower-bound}
Let $A\succ l\cdot I_n$ and $\phi_l(A)\leqslant \phi$, 
$B$ a  positive semi-definite random 
matrix satisfying $\mathbb{E}B=I_n$ and Property 
$(MWR)$ with some $p> 1$.\\  Let $\varepsilon \in (0,1)$, if
\begin{equation}\label{condition-phi}
\phi \leqslant \frac{1}{4\left(8C_p\right)^{\frac{1}{p-1}}} \varepsilon^{\frac{p}{p-1}},
\end{equation}
then there exist $l'$ a random variable such that

$$A+B \succ l'\cdot I_n , \quad \phi_{l'}(A+B)\leqslant 
\phi_l(A) \ \  \text{and} \ \  \mathbb{E} l'\geqslant l+ 1-\varepsilon .$$\\
\end{theo}

\begin{proof}[Proof of Theorem~\ref{lower-bound}]
Let $\phi$ satisfying condition (\cite{condition-phi}) of Theorem~\ref{iteration-lower-bound}. 
We start with $A_0=0$ 
and $l_0=-\frac{n}{\phi}$ so that $\phi_{l_0}(A_0)=-\frac{n}{l_0}=\phi$.\\
Applying Theorem\,\ref{iteration-lower-bound}, one can 
find $l_1$ such that $$A_1=A_0+B_1 \succ l_1 \cdot I_n ,\quad  \phi_{l_1}(A_1)
\leqslant \phi_{l_0}(A_0)=\phi$$
and 
$$\mathbb{E}l_1\geqslant l_0 +1-\varepsilon$$
Now apply Theorem\,\ref{iteration-lower-bound} 
conditionally on $A_1$ to find $l_2$ such that $$A_2=A_1+B_2 \succ l_2 \cdot I_n ,
\quad  \phi_{l_2}(A_2)\leqslant \phi_{l_1}(A_1)$$
and 
$$\mathbb{E}_{B_2}l_2\geqslant l_1 +1-\varepsilon .$$
Using Fubini's Theorem we have 
$$ 
\mathbb{E}l_2\geqslant \mathbb{E}l_1 +1-\varepsilon\geqslant l_0 +2(1-\varepsilon )$$
After $N$ steps, we get $\mathbb{E}\lambda_{\min} (A_N)
\geqslant \mathbb{E} l_N \geqslant l_0 +N(1-\varepsilon )$.  Therefore,
$$\mathbb{E} \lambda_{\min} \left(\frac{1}{N}\displaystyle 
\sum_{i=1}^N B_i \right) \geqslant 1-\varepsilon -\frac{n}{N\phi} .$$
Taking $N=\frac{n}{\varepsilon \phi}$, 
we get $\displaystyle \mathbb{E} \lambda_{\min} \left(\frac{1}{N}\sum_{i=1}^N B_i \right) 
\geqslant 1-2\varepsilon$.

 \end{proof}

\hfill

\section{Proof of Theorem~\ref{upper-bound}}\hfill

Given $A$ an $n\times n$  positive 
semi-definite matrix such that all 
eigenvalues of $A$ are less than an 
upper barrier $u_A= u$ i.e. $A \prec u\cdot I_n$, define 
the corresponding potential function to be $$\psi_u(A) =
{\rm Tr}\left(u\cdot I_n-A\right)^{-1}.$$
The proof of Theorem~\ref{upper-bound} is based on 
the following result which will be proved in section 7:

\begin{theo}\label{iteration-upper-bound}
Let $A\prec u\cdot I_n$ and $\psi_u(A)\leqslant \psi$, 
$B$ a  positive semi-definite random 
matrix satisfying $\mathbb{E}B=I_n$ and Property $(MSR)$.\\  
Let $\varepsilon \in (0, 1)$, if 
\begin{equation}\label{choice-of-psi}
\psi \leqslant {C_3}\cdot\footnote{$C_3 = \left[8(2c)^{\frac{1}{\eta}} (16+ \frac{16}{\eta})^{1+\frac{3}{\eta}}\right]^{-1} \wedge  
\left[ 16 (2c)^{\frac{3}{2} +\frac{2}{\eta}} (8+ \frac{8}{\eta})^{\frac{4}{\eta}} \right]^{-1} $ }\varepsilon^{1+\frac{2}{\eta}}
\end{equation}
 there exists $u'$ a random variable such that
$$A+B \prec u'\cdot I_n , \quad \psi_{u'}(A+B)
\leqslant \psi_u(A) \ \  \text{and} \ \  \mathbb{E} u'\leqslant u+1+\varepsilon  .$$\\
\end{theo}

\begin{proof}[Proof of Theorem~\ref{upper-bound}]
Let $\psi$ satisfying the condition of Theorem~\ref{iteration-upper-bound}. 
We start with $A_0=0$, $u_0=\frac{n}{\psi}$ 
so that $\psi_{u_0}(A_0)=\psi$.\\
Applying Theorem\,\ref{iteration-upper-bound}, one can find $u_1$ such that

$$A_1=A_0+B_1 \prec u_1\cdot I_n , \quad \psi_{u_1}(A_1)\leqslant 
\psi_{u_0}(A_0) \ \  \text{and} \ \  \mathbb{E} u_1\leqslant u_0+1+\varepsilon  .$$

Now apply Theorem\,\ref{iteration-upper-bound} conditionally on $A_1$ to find $u_2$ such that 
$$A_2=A_1+B_2 \prec u_2\cdot I_n , \quad \psi_{u_2}(A_2)
\leqslant \psi_{u_1}(A_1) \ \  \text{and} \ \  \mathbb{E}_{B_2} u_2
\leqslant u_1+1+\varepsilon  .$$
Using Fubini's Theorem we have
$$
\mathbb{E} u_2
\leqslant \mathbb{E}u_1+1+\varepsilon \leqslant u_0 + 2(1+\varepsilon) .$$
After $N$ steps we get $\mathbb{E}\lambda_{\max}
\left(\displaystyle \sum_{i=1}^N B_i\right)\leqslant \mathbb{E} u_N
\leqslant u_0 +N(1+\varepsilon )$.

Taking $N\geqslant \frac{n}{\varepsilon\psi} = C_3^{-1}\frac{n}{\varepsilon^{2+\frac{2}{\eta}}}$, 
we deduce that $$\displaystyle\mathbb{E} \lambda_{\max}
\left(\frac{1}{N}\sum_{i=1}^N B_i\right) \leqslant 1+2\varepsilon.$$
\end{proof}

\section{Proof of Theorem~\ref{iteration-lower-bound}}\hfill

\subsection{Notations}\hfill\\

We work under the assumptions of Theorem~\ref{iteration-lower-bound}. 
We are looking for a random variable $l'$ of the form $l+\delta$ 
where $\delta$ is a positive random variable playing the role of the shift.

If in addition $A\succ (l+\delta)\cdot I_n$, we will note 
$$L_\delta 
=A-(l+\delta)\cdot I_n$$ so that 
$${\rm Tr}\left(B^{\frac{1}{2}}(A-(l+\delta)\cdot I_n)^{-1}B^{\frac{1}{2}}\right)=\left<L_{\delta}^{-1},B\right>.$$

$\lambda_1,..,\lambda_n$ will denote the eigenvalues of 
$A$ and $v_1,..,v_n$ the corresponding normalized eigenvectors. 
Note that $(v_i)_{i\leqslant n}$ are also the eigenvectors of 
$L_{\delta}^{-1}$ corresponding to the eigenvalues 
$\frac{1}{\lambda_i -(l+\delta)}$.\\

\subsection{Finding the shift}\hfill\\

To find sufficient conditions for such $\delta$ 
to exist, we need a matrix extension of Lemma~3.4 in \cite{batson-spielman-srivastava} which, up to a minor change, 
is essentially contained in Lemma\,$20$ in 
\cite{silva-harvey-sato} and we formulate it here in Lemma~\ref{lem-lower-shift}. This method uses 
the Sherman-Morrison-Woodbury formula:

\begin{lem}\label{sherman-morrison}
Let $E$ be an $n\times n$ invertible matrix,
 $C$ a $k \times k$ invertible matrix, $U$ 
 an $n\times k$ matrix and $V$ a $k\times n$ matrix. Then we have:
$$ (E+UCV)^{-1} = E^{-1} - E^{-1}U(C^{-1}+VE^{-1}U)^{-1}VE^{-1}$$
\end{lem}

\begin{lem}\label{lem-lower-shift}
Let $A$ as above satisfying $A\succ l\cdot I_n$. 
Suppose that one can find $\delta >0$ 
verifying $\delta \leqslant \frac{1}{\Vert L_0^{-1}\Vert}$ and

$$\frac{\left<L_{\delta}^{-2},B\right>}{\phi_{l+\delta}(A)
 - \phi_l(A)} - \left\Vert B^{\frac{1}{2}}L_{\delta}^{-1}B^{\frac{1}{2}}\right\Vert\ \geqslant 1$$
 Then 
$$\lambda_{\min}(A+B)> l+\delta \quad and
 \quad \phi_{l+\delta}(A+B) \leqslant \phi_l(A).$$

\end{lem}

\begin{proof}
First note that $\frac{1}{\Vert L_0^{-1}\Vert} = \lambda_{\min} (A)-l$, 
so the first condition on $\delta$ 
implies that $\lambda_{\min}(A)\geqslant l+\delta$.\\
Now using Sherman-Morrison-Woodbury formula with $E=L_{\delta},\ 
U=V=B^{\frac{1}{2}},\ C=I_n$ we get :

\begin{align*}
\phi_{l+\delta}(A+B) &= {\rm Tr}\left(L_{\delta}+B\right)^{-1}\\
&= \phi_{l+\delta}(A) - {\rm Tr}\left(L_{\delta}^{-1}B^{\frac{1}{2}}
\left(I_n+B^{\frac{1}{2}}L_{\delta}^{-1}B^{\frac{1}{2}}\right)^{-1}
B^{\frac{1}{2}}L_{\delta}^{-1}\right)\\
&\leqslant \phi_{l+\delta}(A)-\frac{\left<L_{\delta}^{-2},
B\right>}{1+\left\Vert B^{\frac{1}{2}}L_{\delta}^{-1}B^{\frac{1}{2}}\right\Vert}
\end{align*}
Rearranging the hypothesis, we get $\phi_{l+\delta}(A+B)\leqslant \phi_l(A)$.

\end{proof}

Since $\Vert L_0^{-1}\Vert\leqslant {\rm Tr}\left(L_0^{-1}\right)=\phi_l(A)$ and 
$\left\Vert B^{\frac{1}{2}}L_{\delta}^{-1}B^{\frac{1}{2}}\right\Vert\leqslant  \left<L_\delta^{-1},B\right>$ 
then in order to satisfy conditions of Lemma\,\ref{lem-lower-shift}, we may search 
for $\delta$ satisfying:

\begin{equation}\label{eq-firstcondition-lower-bound}
\delta\leqslant \frac{1}{\phi_l(A)} \quad \text{ and } \quad \frac{\left<L_{\delta}^{-2},B\right>}{\phi_{l+\delta}(A)- \phi_l(A)}-\left<L_\delta^{-1},B\right>\geqslant 1
  \end{equation}

For $t \leqslant \frac{1}{\phi}$, let us note : \vskip 0.2 cm

$$q_1(t , B)= \left<L_t^{-1},B\right>= {\rm Tr}\left(B(A-(l+t)\cdot I_n)^{-1}\right)$$
and
$$q_2(t ,B)= \frac{\left<L_t^{-2},B\right>}{ {\rm Tr}(L_t^{-2})}=
\frac{{\rm Tr}\left(B(A-(l+t)\cdot I_n\right)^{-2})}{{\rm Tr}\left(A-(l+t)\cdot I_n\right)^{-2}}$$ 

We have already seen in Lemma\,\ref{lem-lower-shift} 
that if $t\leqslant \frac{1}{\phi}\leqslant \frac{1}{\Vert L_0^{-1}\Vert}$ then $A\succ (l+t)\cdot I_n$ so the definitions above make sense.
Since we have 

\begin{align*}
\phi_{l+\delta}(A)-\phi_l(A)
&={\rm Tr}(A-(l+\delta)\cdot I_n)^{-1}-{\rm Tr}(A-l\cdot I_n)^{-1} \\
 &= \delta {\rm Tr}((A-(l+\delta)\cdot I_n)^{-1}(A-l\cdot I_n)^{-1})\\
&\leqslant \delta {\rm Tr}(A-(l+\delta)\cdot I_n)^{-2},
\end{align*}
then in order to have $(\ref{eq-firstcondition-lower-bound})$, it will be sufficient to choose $\delta$ 
satisfying $\delta \leqslant \frac{1}{\phi}$ and 
\begin{equation}\label{eq-condition-lower-bound}
\frac{1}{\delta}
 q_2(\delta ,B) - q_1(\delta ,B) \geqslant 1
 \end{equation}
 
Note that $q_1$ and $q_2$ can be expressed as follows: 
$$
q_1(t , B)=\displaystyle \sum_{i=1}^n \frac{\left<Bv_i,v_i\right>}{\lambda_i-l-t} 
\quad \text{ and } \quad q_2(t ,B)=\frac{\sum_{i} 
\frac{\left<Bv_i,v_i\right>}{(\lambda_i-l-t)^2}}{ \sum_{i} (\lambda_i-l-t)^{-2}}.
$$

Since $\phi_l(A)=\displaystyle \sum_{i=1}^n (\lambda_i -l)^{-1} \leqslant \phi$, then $(\lambda_i -l)\cdot \phi 
\geqslant 1$ for all $i$. Thus

$$(1-t\cdot \phi)(\lambda_i-l)=\lambda_i -l -t\cdot (\lambda_i -l)\cdot \phi 
\leqslant \lambda_i -l -t \leqslant \lambda_i -l$$
and therefore $$q_1(t,B) \leqslant (1-t\cdot \phi)^{-1} q_1(0,B)$$ 
and $$(1-t\cdot\phi)^{2}q_2(0,B) \leqslant q_2(t,B) \leqslant (1-t\cdot\phi)^{-2} q_2(0,B).$$

\begin{lem}\label{choice-lower-shift}
Let $s\in (0,1)$ and take $\delta = (1-s)^3q_2(0,B)\mathbf{1}_{\{q_1(0,B)\leqslant s\}} 
\mathbf{1}_{\{q_2(0,B)\leqslant \frac{s}{\phi}\}}$. Then $A+B\succ (l+\delta )\cdot I_n$ 
and $\phi_{l+\delta }(A+B)\leqslant \phi_l(A)$.
\end{lem}

\begin{proof}
As stated before in (\ref{eq-condition-lower-bound}), it is sufficient to check that 
$\delta \leqslant \frac{1}{\phi}$ and $\frac{1}{\delta} 
q_2(\delta ,B) - q_1(\delta ,B) \geqslant 1$.\\
If $q_1(0,B)\geqslant s$ or $q_2(0,B)\geqslant 
\frac{s}{\phi}$ then $\delta =0$ and there is nothing 
to prove since $\phi_{l }(A+B)\leqslant \phi_l(A)$.\\

In the other case i.e. $q_1(0,B)\leqslant  s$ and 
$q_2(0,B)\leqslant  \frac{s}{\phi}$, we have 
$\delta = (1-s)^3q_2(0,B)$. \\So $\delta \leqslant (1-s)^3\frac{s}{\phi}\leqslant \frac{1}{\phi}$ and

\begin{align*}
\frac{1}{\delta} q_2(\delta ,B) - q_1(\delta ,B)
&=\frac{1}{(1-s)^3q_2(0,B)} q_2(\delta ,B) - q_1(\delta ,B)\\
&\geqslant \frac{1}{(1-s)^3q_2(0,B)} (1-\delta \phi )^{2}q_2(0 ,B) - (1-\delta \phi)^{-1}q_1(0 ,B)\\
&\geqslant \frac{(1-s)^2}{(1-s)^3}- \frac{s}{(1-s)}= 1.
\end{align*}
\end{proof}

\subsection{Estimating the random shift}\hfill\\

Now that we have found $\delta$, we will estimate 
$\mathbb{E} \delta$ using the property $(MWR)$. 
We will start by stating some basic facts 
about $q_1$ and $q_2$.

\begin{prop}\label{propriété-q}
Let as above $A\succ l\cdot I_n$ and $\phi_l(A)\leqslant \phi$, 
$B$ satisfying $(MWR)$ and $\mathbb{E}B=I_n$. Then we have the following :
\begin{enumerate}
  
  \item $\mathbb{E}q_1(0,B)= \phi_l(A) 
  \leqslant \phi$ and $\mathbb{E}q_1(0,B)^p 
  \leqslant C_p \phi^p$.
  
  \item $\mathbb{E}q_2(0,B)= 1 $ 
  and $\mathbb{E}q_2(0,B)^p \leqslant C_p$.
  
  \item $\mathbb{P}(q_1(0,B) \geqslant u) 
  \leqslant C_p(\frac{\phi}{u})^p$ and 
  $\mathbb{P}(q_2(0,B)\geqslant u)\leqslant \frac{C_p}{u^p}$.

\end{enumerate}
\end{prop}

\begin{proof}
Since $\mathbb{E}B=I_n$ then 
$\mathbb{E} q_1(0,B)= \phi_l(A)$ and $\mathbb{E}q_2(0,B)=1$.\\
Now using the triangle inequality and Property $(MWR)$ we get :

$$
\left(\mathbb{E}q_1(0,B)^p\right)^{\frac{1}{p}}
=\left[\mathbb{E}\left(\displaystyle \sum_{i=1}^n \frac{\left<Bv_i,v_i\right>}{\lambda_i -l}\right)^p\right]^{\frac{1}{p}}
\leqslant \sum_{i=1}^n \frac{\left( \mathbb{E}\left<Bv_i,v_i\right>^p\right)^{\frac{1}{p}}}{\lambda_i -l}
\leqslant \sum_{i=1}^n \frac{C_p^{\frac{1}{p}}}{\lambda_i -l}
\leqslant C_p^{\frac{1}{p}}\phi
$$

With the same argument we prove that 
$\mathbb{E}q_2(0,B)^p \leqslant C_p$. 
The third part of the proposition follows by Markov's inequality.
\end{proof}

\begin{lem}
If $\delta$ is as in Lemma\,\ref{choice-lower-shift}. 
Then $$\mathbb{E} \delta \geqslant (1-s)^3 
\left[ 1- 2C_p\left(\frac{\phi}{s}\right)^{p-1}\right] .$$
\end{lem}

\begin{proof}
Using the above proposition and 
H$\overset{..}{\text{o}}$lder's inequality 
with $\frac{1}{p}+\frac{1}{q}=1$ we get :

\begin{align*}
\mathbb{E} \delta 
&= \mathbb{E} (1-s)^3q_2(0,B)\mathbf{1}_{\{q_1(0,B)\leqslant s\}} 
\mathbf{1}_{\{q_2(0,B)\leqslant \frac{s}{\phi}\}}\\
&= (1-s^3)\left[ \mathbb{E} q_2(0,B) - \mathbb{E}q_2(0,B) 	
\mathbf{1}_{\{q_1(0,B)> s\ or\  q_2(0,B)> \frac{s}{\phi}\}} \right]\\
&\geqslant (1-s)^3 \left[ 1-  \left(\mathbb{E}q_2(0,B)^p\right)^{\frac{1}{p}}. 
\left(\mathbb{P}\left\{q_1(0,B)> s\ or\  q_2(0,B)> \frac{s}{\phi}\right\}\right)^{\frac{1}{q}}\right]\\
&\geqslant (1-s)^3 \left[ 1- C_p^{\frac{1}{p}} \left( C_p\left(\frac{\phi}{s}\right)^p 
+ C_p\left(\frac{\phi}{s}\right)^p \right)^{\frac{1}{q}} \right]\\
&\geqslant (1-s)^3 \left[ 1- 2C_p\left(\frac{\phi}{s}\right)^{p-1}\right] .
\end{align*}

\end{proof}

Now it remains to make good choice of $s$ and 
$\phi$ in order to finish the prove Theorem~\ref{iteration-lower-bound}. 
Take $l'=l+\delta$, the choice of $\delta$ 
being as before with $s=\frac{\varepsilon}{4}$. \\
As we have seen, we get $A+B \succ l'\cdot I_n$ and 
$\phi_{l'}(A+B)\leqslant \phi_l(A)$. Moreover,
$$
\mathbb{E}l' = l+ \mathbb{E}\delta \geqslant 
l+(1-s)^3\left[ 1- 2C_p\left(\frac{\phi}{s}\right)^{p-1}\right]
\geqslant 1-\varepsilon,
$$
 by the assumption on $\phi$. This ends the proof of Theorem~\ref{iteration-lower-bound}.

\hfill

\section{Proof of Theorem~\ref{iteration-upper-bound}}\hfill

\subsection{Notations}\hfill\\

We work under the assumptions of Theorem~\ref{iteration-lower-bound}. 
We are looking for a random variable $u'$ of the form $u+\Delta$ 
where $\Delta$ is a positive random variable playing the role of the shift.

We will denote 
$$U_t = (u+t)\cdot I_n- A , $$ so that 
$${\rm Tr}\left(B^{\frac{1}{2}} ((u+t)\cdot I_n -A)^{-1}B^{\frac{1}{2}}\right)= \left<{U_t}^{-1},B\right>.$$

As before, $\lambda_1,..,\lambda_n$ will denote the 
eigenvalues of $A$ and $v_1,..,v_n$ the corresponding 
normalized eigenvectors. $(v_i)_{i\leqslant n}$ are also the eigenvectors
 of $U_t^{-1}$ corresponding to the eigenvalues $\frac{1}{u+t-\lambda_i}$.

\subsection{Finding the shift}\hfill\\

To find sufficient conditions for such $\Delta$ to exist, 
we need a matrix extension of Lemma~3.3 in \cite{batson-spielman-srivastava} which, up to a minor change, 
is essentially contained in Lemma\,$19$ in 
\cite{silva-harvey-sato}.
For the sake of completeness, we include the proof.

\begin{lem}\label{lem-condition-upper-bound} 
Let $A$ as above satisfying $A\prec u\cdot I_n$. 
Suppose that one can find $\Delta > 0$ verifying
\begin{equation}\label{condition-upper-bound}
\frac{\left<U_{\Delta}^{-2},B\right>}{\psi_u(A)-
\psi_{u+\Delta}(A)} + \left\Vert B^{\frac{1}{2}}U_{\Delta}^{-1}B^{\frac{1}{2}}\right\Vert \leqslant 1.
\end{equation}
Then $$A+B \prec (u+\Delta)\cdot I_n \quad \text{and} 
\quad \psi_{u+\Delta}(A+B) \leqslant \psi_u(A).$$
\end{lem}

\begin{proof}
Since $\left<U_{\Delta}^{-2},B\right>$ and $\psi_u(A)-\psi_{u+\Delta}(A)$ 
are positive, then by $(\ref{condition-upper-bound})$ we have 

$$\left\Vert B^{\frac{1}{2}}U_{\Delta}^{-1}B^{\frac{1}{2}}\right\Vert < 1\quad \text{and}\quad 
\frac{\left<U_{\Delta}^{-2},B\right>}{1- \left\Vert B^{\frac{1}{2}}U_{\Delta}^{-1}B^{\frac{1}{2}}\right\Vert } 
\leqslant \psi_u(A)-\psi_{u+\Delta}(A).$$

First note that $\left\Vert B^{\frac{1}{2}}U_{\Delta}^{-1}B^{\frac{1}{2}}\right\Vert = 
\left\Vert U_{\Delta}^{-\frac{1}{2}}BU_{\Delta}^{-\frac{1}{2}}\right\Vert< 1$, \ so \ 
$U_{\Delta}^{-\frac{1}{2}}BU_{\Delta}^{-\frac{1}{2}} \prec I_n$. 
Therefore we get $B\prec U_{\Delta}$ 
which means that $A+B\prec (u+\Delta)\cdot I_n$.\\ 
Now using the Sherman-Morrison-Woodbury (see Lemma\,\ref{sherman-morrison}) 
with $E=U_{\Delta}, U=V= B^{\frac{1}{2}}, C=I_n$ we get :

\begin{align*}
 \psi_{u+\Delta}(A+B)
 &= {\rm Tr}\left(U_{\Delta}-B\right)^{-1}\\
&=\psi_{u+\Delta}(A)+ {\rm Tr}\left(U_{\Delta}^{-1}B^{\frac{1}{2}}
\left(I_n-B^{\frac{1}{2}}U_{\Delta}^{-1}B^{\frac{1}{2}}\right)^{-1}B^{\frac{1}{2}}U_{\Delta}^{-1}\right)\\
&\leqslant \psi_{u+\Delta}(A) + 
\frac{\left<U_{\Delta}^{-2},B\right>}{1-\left\Vert U_{\Delta}^{-\frac{1}{2}}BU_{\Delta}^{-\frac{1}{2}}\right\Vert}\leqslant \psi_u(A)
\end{align*}
\end{proof}

We may now find $\Delta$ satisfying $(\ref{condition-upper-bound})$. 
Let us note :

$$Q_1(t,B)= \left\Vert B^{\frac{1}{2}}U_t^{-1}B^{\frac{1}{2}}\right\Vert = \left\Vert B^{\frac{1}{2}}\left((u+t)\cdot I_n-A\right)^{-1}B^{\frac{1}{2}}\right\Vert$$
and
$$Q_2(t,B) = \frac{\left<U_t^{-2},B\right>}{\psi_u(A)-\psi_{u+t}(A)}= 
\frac{{\rm Tr}\left(B\left((u+t)\cdot I_n-A\right)^{-2}\right)}{\psi_u(A)-\psi_{u+t}(A)}.$$

Since $Q_1$ and $Q_2$ are both decreasing in $t$,
 we work with each separately. Note that by (\ref{choice-of-psi}), we have $4c\cdot \psi \leqslant 1/2$. 
 Fix $\theta \in(4c\cdot \psi,1/2)$ 
 and define $\Delta_1, \Delta_2$ as follows :

$$\Delta_1  \text{ the smallest positive number such that } 
Q_1(\Delta_1,B)\leqslant \theta$$
and
$$\Delta_2\text{ the smallest positive number such that } 
Q_2(\Delta_2,B)\leqslant 1-\theta .$$

Now take $\Delta = \Delta_1 + \Delta_2$, then 
$Q_1(\Delta,B)+ Q_2(\Delta,B) \leqslant \theta +1-\theta =1$. 
So this choice of $\Delta$ satisfies $(\ref{condition-upper-bound})$ 
and it remains now to estimate $\Delta_1$ and $\Delta_2$ separately.

\subsection{Estimating $\Delta_1$}\hfill\\

We may 
write $Q_1(\Delta_1,B) = \left\Vert\displaystyle \sum_{i=1}^n \frac{B^{\frac{1}{2}}v_iv_i^tB^{\frac{1}{2}}}{u+\Delta_1 -\lambda_i}\right\Vert$. \\

Put $\xi_i= B^{\frac{1}{2}}v_iv_i^tB^{\frac{1}{2}}$, $\mu_i= \psi (u-\lambda_i)$ and 
$\mu =\psi \Delta_1$. Denote $P_S$ the orthogonal projection on ${\rm span}\left(\{v_i\}_{i\in S}\right)$, 
clearly $rank(P_S)= \left\vert S\right\vert$. Then we have :\\

$\left\{
\begin{array}{ll}
\displaystyle
\mathbb{E}\Vert \xi_i\Vert= \mathbb{E} \langle Bv_i, v_i\rangle =1. \\ \\
\mathbb{P}\left(\left\Vert\displaystyle \sum_{i\in S} \xi_i \right\Vert\geqslant t\right)= \mathbb{P} \left(\left\Vert P_SBP_S\right\Vert\geqslant t\right)
\leqslant \frac{c}{t^{1+\eta}} \quad \forall t\geqslant c\left\vert S\right\vert .\\ \\
\displaystyle \sum_{i=1}^n \frac{1}{\mu_i}=\frac{\psi_u(A)}{\psi}\leqslant 1.\\ \\
\mu \text{ is the smallest positive number such that } \displaystyle 
\sum_{i=1}^n \frac{\xi_i}{\mu_i+\mu} \preceq \frac{\theta}{\psi}I_n .\\
\end{array}
\right.\\$

We will need an analog of Lemma\,3.5 appearing in \cite{srivastava-vershynin}. We extend 
this lemma to a matrix setting:

 \begin{lem}\label{lemme-MSR}
 Suppose $\{ \xi_i\}_{i\leqslant n}$ are symmetric positive semi-definite random 
 matrices with $\mathbb{E} \Vert \xi_i\Vert =1$ and 
 $$ \mathbb{P} \left( \left\Vert  \sum_{i\in S} \xi_i\right\Vert \geqslant t \right)
  \leqslant \frac{c}{t^{1+\eta}} \quad provided \quad t>c\vert S\vert 
  = c\sum_{i\in S} \mathbb{E} \Vert \xi_i\Vert ,$$
 for all subsets $S\subset [n]$ and some constants $c,\eta >0$.
  Consider positive numbers $\mu_i$ such that
 $$\displaystyle \sum_{i =1}^n \frac{1}{\mu_i} \leqslant 1.$$
 Let $\mu$ be the minimal positive number such that 
 $$\displaystyle \sum_{i=1}^n \frac{\xi_i}{\mu_i +\mu} \preceq K\cdot I_n,$$
 for some $K\geqslant C=4c$.\quad Then \quad
 $\mathbb{E} \mu \leqslant \frac{c(\eta)}{K^{1+\eta}}.$
 \end{lem}
  \begin{proof}[Proof]
 For any $j\geqslant 0$, denote 
 $$
 I_j=\left\{i/\ 2^j\leqslant \mu_i <2^{j+1}\right\}
 $$
 and let $n_j= \left\vert I_j\right\vert$. Note that 
 $$
 \sum_{j=0}^{+\infty} \frac{n_j}{2^{j+1}}\leqslant
  \sum_{j=0}^{+\infty} \sum_{i\in I_j} \frac{1}{\mu_i} =
  \sum_{i\leqslant n} \frac{1}{\mu_i}\leqslant 1.
 $$
 Define $\mu'$ as the minimal positive number such that 
 $$
 \forall j\geqslant 0, \quad \quad \frac{1}{2^j +\mu'}\left\Vert 
 \sum_{i\in I_j}\xi_i\right\Vert \leqslant \varepsilon_j,
 $$
 where $\varepsilon_j = \frac{K}{2} \frac{n_j}{2^{j+1}} 
 \vee \frac{K}{2a}2^{-j\frac{\eta}{2+2\eta}}$ and 
 $a = \sum_{j} 2^{-j\frac{\eta}{2+2\eta}}$.
 First note that $\mu\leqslant \mu'$; indeed, 
 \begin{align*}
\left\Vert  \sum_{i=1}^n \frac{\xi_i}{\mu_i +\mu'}\right\Vert
= \left\Vert  \sum_{j=0}^{+\infty}\sum_{i\in I_j} \frac{\xi_i}{\mu_i +\mu'}\right\Vert
&\leqslant \left\Vert\sum_{j=0}^{+\infty}  \frac{1}{2^j +\mu'}\sum_{i\in I_j} \xi_i\right\Vert \\
&\leqslant \sum_{j=0}^{+\infty} \frac{1}{2^j +\mu'}\left\Vert \sum_{i\in I_j} \xi_i\right\Vert \\
&\leqslant \sum_{j=0}^{\infty} \varepsilon_j \\
&\leqslant \frac{K}{2}+\frac{K}{2}=K,
 \end{align*}
and since $\mu$ is the minimal positive number satisfying the inequality above, then $\mu\leqslant\mu'$. 
We may now estimate $\mathbb{E}\mu'$; to this aim, we need to look at $\mathbb{P}\left\{\mu'\geqslant t\right\}$. 
For $t\geqslant 0$, 
\begin{align*}
\mathbb{P}\left\{\mu'\geqslant t\right\}
&= \mathbb{P}\left\{\exists j\geqslant 0/\ \frac{1}{2^j +t}\left\Vert 
 \sum_{i\in I_j}\xi_i\right\Vert \geqslant \varepsilon_j\right\}\\
 &\leqslant \sum_{j=0}^{+\infty}  \mathbb{P}\left\{\frac{1}{2^j +t}\left\Vert 
 \sum_{i\in I_j}\xi_i\right\Vert \geqslant \varepsilon_j\right\}\\
 &= \sum_{j=0}^{+\infty}  \mathbb{P}\left\{\left\Vert 
 \sum_{i\in I_j}\xi_i\right\Vert \geqslant \varepsilon_j(2^j+t)\right\}\\
 &\leqslant \sum_{j=0}^{+\infty}  \frac{c}{\left[\varepsilon_j(2^j+t)\right]^{1+\eta}},
\end{align*}
where the last inequality comes from the fact that $\varepsilon_j(2^j+t)\geqslant 
\frac{K}{4}n_j\geqslant c\vert I_j\vert$ and by applying the hypothesis satisfied by 
the $\xi_i$. Now since $\varepsilon_j\geqslant \frac{K}{2a}2^{-j\frac{\eta}{2+2\eta}}$, 
we have 
$$
\mathbb{P}\left\{\mu'\geqslant t\right\}
\leqslant \left(\frac{2a}{K}\right)^{1+\eta}\sum_{j=0}^{+\infty} 
\frac{c}{2^{-j\frac{\eta}{2}}\left(2^j+t\right)^{1+\eta}}
\leqslant \frac{c(\eta)}{K^{1+\eta}} \sum_{j=0}^{+\infty}\frac{1}{(2^j+t)^{1+\frac{\eta}{2}}} .
$$
Now by integration we get,
$$
\mathbb{E}\mu' = \int_0^{+\infty} \mathbb{P}\left\{\mu'\geqslant t\right\} dt
\leqslant \frac{c(\eta)}{K^{1+\eta}} \sum_{j=0}^{+\infty}\int_0^{+\infty}\frac{1}{(2^j+t)^{1+\frac{\eta}{2}}}dt
\leqslant \frac{c(\eta)}{K^{1+\eta}} .
$$

 \end{proof}

 Applying Lemma\,\ref{lemme-MSR} we get 
 $\mathbb{E} \mu \leqslant c(\eta)  
  \left(\frac{\psi}{\theta}\right)^{1+\eta}$, so that 
  \begin{equation}\label{eq-delta1}
 \mathbb{E} \Delta_1 \leqslant c(\eta)\frac{\psi^{\eta}}{\theta^{1+\eta}}.
 \end{equation}

 \subsection{Estimating $\Delta_2$}\hfill\\

 Since $$ \psi_{u}(A)- \psi_{u+t}(A)=t\cdot {\rm Tr}\left( \left(u\cdot I_n-A\right)^{-1}\left((u+t)\cdot I_n-A\right)^{-1}\right),$$ 
 we can write 
\begin{align}\label{def-P_2}
 Q_2(t,B) = \frac{ \sum_i\frac{\left< Bv_i,v_i\right>}{(u+t -\lambda_i)^2}}{t\sum_i (u+t-\lambda_i)^{-1}(u-\lambda_i)^{-1}} 
 &\leqslant \frac{1}{t}\frac{ \sum_i\frac{\left< Bv_i,v_i\right>}{(u+t -\lambda_i)(u-\lambda_i)}}{\sum_i (u+t-\lambda_i)^{-1}(u-\lambda_i)^{-1}}\\
 &:=\frac{1}{t}P_2(t,B) .\nonumber
 \end{align}
First note that $P_2(t,B)$ can be written as $\sum_i \alpha_i(t) \left<Bv_i,v_i\right>$ with $\sum_i \alpha_i(t)=1$. 
Having this in mind, one can easily check that $\mathbb{E} P_2(t,B)=1$ and 
\begin{equation}\label{eq-moment-P2}
\mathbb{E} P_2(t,B)^{1+\frac{3\eta}{4}} 
  \leqslant c(\eta),
 \end{equation}
where for the last inequality, we used the fact that $B$ satisfies $(MWR)$ with $p=1+\frac{3\eta}{4}$.\\
In order to estimate $\Delta_2$, we will divide it into two parts as follows:
$$
\Delta_2= \Delta_2 \mathbf{1}_{\{P_2(0,B)\leqslant \frac{\theta}{4\psi}\}} +
 \Delta_2 \mathbf{1}_{\{P_2(0,B)> \frac{\theta}{4\psi} \}}:= H_1 +H_2 .
 $$
  Let us start by estimating $\mathbb{E}H_1$. Suppose that $P_2(0,B) \leqslant \frac{\theta}{4\psi}$ and 
  denote $$x=(1+4\theta) P_2(0,B).$$ Since $\psi_u(A)\leqslant \psi$, we have $ (u-\lambda_i).\psi\geqslant 1\ \forall i$ 
  and therefore $
  u+x - \lambda_i \leqslant (1+x\psi)(u-\lambda_i).
  $
  This implies that 
  $$
  P_2(x,B) \leqslant (1+x\psi) P_2(0,B).$$ 
  Now write
  $$
  Q_2(x,B) \leqslant \frac{1}{x}P_2(x,B)\leqslant \frac{1+x\psi}{x}P_2(0,B) 
  \leqslant \frac{1+(1+4\theta)\frac{\theta}{4}}{1+4\theta}
  \leqslant 1-\theta,
  $$
which means that 
$$
\Delta_2 \mathbf{1}_{ \left\{ P_2(0,B)
\leqslant \frac{\theta}{4\psi}\right\}} \leqslant (1+4\theta )P_2(0,B),
$$
and therefore
\begin{equation}\label{eq-delta2-firstpart}
\mathbb{E}H_1 =\mathbb{E}\Delta_2 \mathbf{1}_{ \left\{ P_2(0,B)
\leqslant \frac{\theta}{4\psi}\right\}} \leqslant 1+4\theta .
\end{equation}
  
Now it remains to estimate $\mathbb{E}H_2$. For that we need 
to prove a moment estimate for $\Delta_2$. First observe that using (\ref{eq-moment-P2}) we have
$$
\mathbb{P}\{\Delta_2 >t\} = \mathbb{P}\{Q_2(t,B)>1-\theta\}
   \leqslant \mathbb{P}\{P_2(t,B)>t.(1-\theta)\}\leqslant \frac{c(\eta)}{t^{1+\frac{3\eta}{4}}}.
$$
By integration, this implies 
$$
\mathbb{E} \Delta_2^{1+\frac{\eta}{2}} =
   \int_0^\infty \mathbb{P} \{\Delta_2 >t\} (1+\frac{\eta}{2})t^{\frac{\eta}{2}}dt
\leqslant \int_0^1 (1+\frac{\eta}{2})t^{\frac{\eta}{2}}dt + \int_1^\infty \frac{c(\eta)}{t^{1+\frac{\eta}{4}}}
\leqslant c(\eta).
$$

Let $p'=1+\frac{\eta}{2}$, applying H$\overset{..}{\text{o}}$lder's 
inequality with $\frac{1}{p'}+\frac{1}{q'} =1$ we have :

\begin{align}\label{eq-delta2-secondpart}
\mathbb{E}H_2=\mathbb{E} \Delta_2  \mathbf{1}_{ \left\{ P_2(0,B)> \frac{\theta}{4\psi}\right\}} 
&\leqslant \left(\mathbb{E}\Delta_2^{p'}\right)^{\frac{1}{p'}}\left(\mathbb{P}\left\{P_2(0,B)> 
\frac{\theta}{4\psi}\right\}\right)^{\frac{1}{q'}}\\
&\leqslant c(\eta) \left(\left(\frac{\psi}{\theta}\right)^{1+\frac{\eta}{2}}\mathbb{E}P_2(0,B)^{1+\frac{\eta}{2}}\right)^{\frac{1}{q'}}\nonumber\\
&\leqslant c(\eta)\left(\frac{\psi}{\theta}\right)^{\frac{\eta}{2}}.\nonumber
\end{align}

Looking at (\ref{eq-delta2-firstpart}) and (\ref{eq-delta2-secondpart}) we have
$$
\mathbb{E}\Delta_2\leqslant 1+4\theta +c(\eta)\left(\frac{\psi}{\theta}\right)^{\frac{\eta}{2}}.
$$

 Putting the estimates of $\Delta_1$ and $\Delta_2$ together we deduce

$$\mathbb{E} \Delta \leqslant 1+4\theta +c(\eta)\left(\frac{\psi}{\theta}\right)^{\frac{\eta}{2}}
+c(\eta)\frac{\psi^{\eta}}{\theta^{1+\eta}} .$$

\vskip 0.2cm
We are now ready to finish the proof. 
Take $u' = u+\Delta$ with $\theta=\frac{\varepsilon}{8}$.
 Then taking $\psi= c(\eta) \varepsilon^{1+\frac{2}{\eta}}$ with the constant 
 depending on $c$ and $\eta$ properly chosen, we get $\mathbb{E}\Delta \leqslant 1+\varepsilon$.

\hfill

\section{Isotropic log-concave matrices}\hfill

A natural way to define a log-concave matrix 
is to ask that it has a log-concave distribution. 
We will define the isotropic condition as follows:

\begin{defi}\label{log-concave-defi}
Let $A$ be an $n\times m$ random matrix and denote 
by $(C_i)_{i\leqslant m}$ its columns. We will say that $A$ 
is an isotropic log-concave matrix if $A'^t= \sqrt{m}(C_1^t,..,C_m^t)$ 
is an isotropic log-concave random vector 
in $\mathbb{R}^{nm}$. 
\end{defi}

\begin{Rq}
Let $(a_{i,j})$ the entries of $A$. Saying that 
$A'$ is isotropic means that 
$$
\mathbb{E} a_{i,j}a_{k,l}= \frac{1}{m}\delta_{(i,j),(k,l)}
$$
This implies that for any $n\times m$ matrix $M$ we have 
$$
\mathbb{E} \left< A,M\right> A=
\mathbb{E}{\rm Tr}\left(A^t M\right)A = \frac{1}{m}M.
$$
  One can view this as an analogue to the isotropic condition in the vector case: in fact if 
  $A=X$ is a vector (i.e an $n\times 1$ matrix), the above condition would be 
  $$
  \mathbb{E}\left<X, y\right>X = y \quad \text{ for all } y\in\mathbb{R}^{n},
  $$ 
  which means that $X$ is isotropic in $\mathbb{R}^n$.
\end{Rq}

In \cite{paouris} and \cite{MR2833584}, Paouris established large deviation inequality and 
small ball probability estimate satisfied by an isotropic log-concave vector. Moreover, Guédon-Milman 
obtained in \cite{MR2846382} what is known as thin-shell estimate for isotropic log-concave vector. 
We will derive analogue properties for isotropic log-concave matrices using the results above.

\begin{prop}\label{log-concave-matrix-concentration}
Let $A$ be an $n\times m$ isotropic log-concave matrix and denote $B=AA^t$.
 Then for every orthogonal projection $P$ on $\mathbb{R}^n$ we have the following large deviation 
 estimate for  ${\rm Tr}(PB)$
\begin{equation}\label{eq-large-deviation-matrix}
\mathbb{P}\left\{ {\rm Tr}(PB)\geqslant c_1t\right\} \leqslant \exp{\left(-\sqrt{t.m}\right)}\quad \forall t\geqslant rank(P)
\end{equation}
and a small ball probability estimate
\begin{equation}\label{eq-small-ball-matrix}
\mathbb{P}\left\{ {\rm Tr}(PB)\leqslant c_2\varepsilon .rank(P)\right\} \leqslant \varepsilon^{c_2\sqrt{m.rank(P)}}\quad \forall \varepsilon\leqslant 1.
\end{equation}
Moreover, we also have a thin-shell estimate
\begin{equation}\label{eq-thin-shell-matrix}
\mathbb{P}\left\{\left\vert{\rm Tr}(PB)-{\rm rank}(P)\right\vert
\geqslant  t.{\rm rank}(P) \right\} \leqslant C\exp{\left(-ct^3\sqrt{m.{\rm rank}(P)}\right)}\quad \forall t\leqslant 1.
\end{equation}
\end{prop}

\begin{proof}[Proof]
Let $P$ be an orthogonal projection on $\mathbb{R}^n$ and 
denote $P'=I_m\otimes P$. As we have seen before ${\rm Tr}(PB)= 
\Vert PA\Vert_{\rm HS}^2 = \frac{1}{m}\Vert P'A'\Vert_2^2$ and 
$rank(P') = m.rank(P)$. Since $P'A'$ is an isotropic 
log-concave vector, then using the large deviation inequality \cite{paouris} satisfied by $P'A'$, we have
$$
\mathbb{P}\left\{ \Vert P'A'\Vert_2^2\geqslant c_1u\right\} \leqslant \exp{\left(-\sqrt{u}\right)}\quad \forall u\geqslant rank(P').
$$ 
Let $t\geqslant rank(P)$ and write $u = t.m$. Since $u\geqslant m.rank(P) =rank(P')$ we have
$$
\mathbb{P}\left\{ m .Tr(P B)\geqslant c_1t.m\right\} \leqslant \exp{\left(-\sqrt{t.m}\right)}
$$
which gives the large deviation estimate stated above.\\

For the small ball probability estimate, we apply Paouris result \cite{MR2833584} 
to $P'A'$:
$$
\mathbb{P}\left\{ \Vert P'A'\Vert_2^2\leqslant c_2\varepsilon .rank(P')\right\} 
\leqslant \varepsilon^{c_2\sqrt{rank(P')}}\quad \forall \varepsilon\leqslant 1.
$$
Writing this in terms of $B$ and $P$, we easily get the conclusion. 
Using the thin-shell estimate obtained in \cite{MR2846382} and following the same procedure as above, 
we get the last part of the proposition.
\end{proof}

\begin{Rq}
Recently, it was shown in \cite{fgp} that $s$-concave random vectors 
satisfy a thin-shell concentration similar to the log-concave one. Therefore, all results of 
this section extend to the case of $s$-concave random matrices.
\end{Rq}

In \cite{srivastava-vershynin}, it was shown that an 
isotropic log-concave vector satisfies $(SR)$ and 
we showed in Proposition~\ref{vector-to-matrix} 
how to pass from $(SR)$ to $(MSR*)$. Therefore, we may 
apply Theorem~\ref{covariance-estimate} to log-concave matrices 
and get the following:

\begin{prop}\label{prop-covariance-log-concave}
Let $A$ be an $n\times m$ isotropic log-concave matrix. 
Then $B=AA^t$ satisfies $(MSR)$. Moreover $\forall \varepsilon \in (0,1)$, 
taking $N > c(\varepsilon)n$ independent copies of $B$ we have
$$\mathbb{E} \displaystyle \left\Vert  \frac{1}{N}
\sum_{i=1}^N B_i -I_n\right\Vert \leqslant \varepsilon.
$$
\end{prop}

\begin{proof}[Proof]
Note first that since $A$ is isotropic in the sense of 
definition~\ref{log-concave-defi}, then $B = AA^t$ 
satisfies $\mathbb{E} B =I_n$.

By proposition~\ref{log-concave-matrix-concentration}, $B$ satisfies 
$$ 
\mathbb{P} \left({\rm Tr}(PB)\geqslant c_1t\right) \leqslant \exp(-\sqrt{tm})
 \quad \forall t\geqslant rank(P) \text{ and $\forall P$ 
 orthogonal projection of $\mathbb{R}^{n}$}.  
 $$
and therefore $(MSR*)$.
Applying Theorem\,\ref{covariance-estimate} we 
deduce the result.
\end{proof}
\ \\

\subsection{Eigenvalues of the empirical sum of a log-concave matrix}\hfill\\

The concentration inequalities satisfied by log-concave matrices will 
allow us obtain some results with high probability rather 
than in expectation as was the case before. Precisely, we can 
prove the following : 

\begin{theo}\label{upper-bound-log-concave}
Let $n,m$ and $N$ some fixed integers. Let $A$ be an $n\times m$ 
isotropic log-concave matrix and denote $B=AA^t$. 
For any $\varepsilon\in (0,1)$, if $m\geqslant \frac{C}{\varepsilon^6}\left[
\log (CnN)\right]^2$, then with probability $\geqslant 1- \exp (-c\varepsilon^3\sqrt{m})$ we have 
$$
 \lambda_{\max}\left(\frac{1}{N}\sum_{i=1}^N B_i\right) \leqslant  1+\varepsilon+\frac{6n}{\varepsilon N},
$$
where $(B_i)_{i\leqslant N}$ are independent copies of $B$.
\end{theo}

\begin{proof}[Proof]

The proof of Theorem\,\ref{upper-bound-log-concave} follows the same 
ideas developed in the previous sections. Let $\varepsilon\in (0,1)$, we only need the following 
property satisfied by our matrix $B= AA^t$:
$$ 
\mathbb{P} \left(\left<Bx,x\right> \geqslant 1+\frac{\varepsilon}{2}\right) 
\leqslant \exp(-c\varepsilon^3\sqrt{m}) \quad \forall x\in \mathbb{S}^{n-1}.
$$
This is obtained by applying (\ref{eq-thin-shell-matrix}) 
for rank $1$ projections and looking only at the large deviation part. 

Define $\Delta$, $\psi$ and $\alpha$ as follows:
$$
\Delta = 1+\varepsilon; \quad \psi =\frac{\varepsilon}{6}
 \quad \text{ and } \quad \alpha =\frac{1+\frac{\varepsilon}{2}}{1+\varepsilon}.
$$
Recall some notations 
$$A_0= 0,\ A_1=B_1,\ A_2=A_1+B_1,\ ..,\ A_N= A_{N-1} +B_N 
=\sum_{i=1}^N B_i.
$$ 
Denote 
$$
u_0= \frac{n}{\psi},\ u_1=u_0+\Delta,\ u_2=u_1+\Delta,\ ..,\ u_N=u_{N-1}+\Delta 
=(1+\varepsilon)N + \frac{6n}{\varepsilon}.
$$  
Define 
$$
\psi_{u_i}(A_i) = {\rm Tr}\left(u_i.I_n-A_i\right)^{-1},
$$ the corresponding 
potential function when $A_i \prec u_i.I_n$. Denote by $\Im_i$ the 
event 
$$
\Im_i :=``A_i\prec u_i.I_n \quad \text{ and }\quad \psi_{u_i}(A_i)\leqslant \psi ".
$$
Clearly $\mathbb{P}\left(\Im_0\right)=1$. Suppose now that $\Im_i$ 
is satisfied; as we have seen in Lemma~\ref{lem-condition-upper-bound}, 
the following condition is 
sufficient for the occurrence of the event $\Im_{i+1}$  :
$$
Q_2(\Delta,B_{i+1}) +Q_1(\Delta,B_{i+1}) \leqslant 1.
$$
Note that 
$$
Q_2(\Delta,B_{i+1})\leqslant \frac{1}{\Delta} P_2(\Delta,B_{i+1}),
$$ where 
$P_2$ is defined in (\ref{def-P_2}) but with $A_i$ instead of $A$. Now denoting $\lambda_j$ the eigenvalues of $A_{i}$ and $v_j$ 
the corresponding eigenvectors, taking the probability with respect to $B_{i+1}$ one can write 
\begin{align*}
&\mathbb{P}_{B_{i+1}}\left\{Q_2(\Delta,B_{i+1}) +Q_1(\Delta,B_{i+1})> 1\right\}
\leqslant
\mathbb{P}_{B_{i+1}}\left\{\frac{1}{\Delta} P_2(\Delta,B_{i+1}) +Q_1(\Delta,B_{i+1})> 1\right\}\\
&\leqslant \mathbb{P}_{B_{i+1}}\left\{\frac{1}{\Delta} P_2(\Delta,B_{i+1}) >\alpha\right\}
+\mathbb{P}_{B_{i+1}}\left\{ Q_1(\Delta,B_{i+1})>1-\alpha\right\}\\
&\leqslant  \displaystyle\mathbb{P}_{B_{i+1}}\left\{\sum_{j=1}^n 
\frac{\left<B_{i+1}v_j,v_j\right>}{(u_{i+1}-\lambda_j)(u_i-\lambda_j)} > \alpha \Delta
\sum_{j=1}^n \frac{1}{(u_{i+1}-\lambda_j)(u_i-\lambda_j)}\right\}\\
&\quad +\mathbb{P}_{B_{i+1}}\left\{ \sum_{j=1}^n \frac{\left<B_{i+1}v_j,v_j\right>}{u_{i+1}-\lambda_j}>1-\alpha\right\}\\
&\leqslant  \displaystyle\mathbb{P}_{B_{i+1}}\left\{\exists j\leqslant n\mid \left<B_{i+1}v_j,v_j\right>  > 1+\frac{\varepsilon}{2} \right\}
+\mathbb{P}_{B_{i+1}}\left\{ \exists j\leqslant n\mid \left<B_{i+1}v_j,v_j\right> >\frac{1-\alpha}{\psi}\right\}\\
&\leqslant 2Cn.\exp(-c\varepsilon^3\sqrt{m}).
\end{align*}
Keeping in mind that $B_i$ are independent, we have shown that $$\mathbb{P}\left(\Im_{i+1}\vert \Im_i\right) 
\geqslant 1-2Cn.\exp(-c\varepsilon^3\sqrt{m}).$$ Moreover, we have 

\begin{align*}
\mathbb{P}\displaystyle \left\{ \lambda_{\max}\left( \frac{1}{N}\sum_{i=1}^N B_i\right) 
\leqslant 1+\varepsilon+\frac{6n}{\varepsilon N}\right\}&\geqslant \mathbb{P}\left( \Im_N \right)\\
 &\geqslant \mathbb{P}\left( \Im_N \vert \Im_{N-1}\right)
 \mathbb{P}\left( \Im_{N-1} \vert \Im_{N-2}\right)\cdot \cdot\cdot\mathbb{P}\left(\Im_0\right)\\
&\geqslant 1-2CNn.\exp(-c\varepsilon^3\sqrt{m}).
\end{align*}
Therefore, Theorem\,\ref{upper-bound-log-concave} follows by the choice of $m$.
\end{proof}

\begin{Rq}\label{Rq-upper-bound-log-concave}
Note that in the previous proof, we only used the large deviation 
inequality given by the thin-shell estimate (\ref{eq-thin-shell-matrix}). 
If one uses the deviation inequality given by (\ref{eq-large-deviation-matrix}), 
then by the same proof it can be proved that 
$$
 \lambda_{\max}\left(\frac{1}{N}\sum_{i=1}^N B_i\right) \leqslant  C(1+\frac{n}{N}),
$$
with high probability and with similar condition on $m$. The advantage of using 
thin-shell is that we can get an estimate close to $1$.
\end{Rq}

By the same techniques, we also get an estimate of the smallest 
eigenvalue.

\begin{theo}\label{lower-bound-log-concave}
Let $n, m$ and $N$ some fixed integers. Let $A$ be an $n\times m$ 
isotropic log-concave matrix and denote $B=AA^t$. 
For any $\varepsilon\in (0,1)$, if $m\geqslant \frac{C}{\varepsilon^6}\left[
\log (CnN)\right]^2$, then 
with probability $\geqslant 1-\exp(-c\varepsilon^3\sqrt{m})$ we have
$$
\lambda_{\min}\left(\frac{1}{N}\sum_{i=1}^N B_i\right) \geqslant 1-\varepsilon -\frac{3n}{\varepsilon N},
$$
where $(B_i)_{i\leqslant N}$ are independent copies of $B$.
\end{theo}

\begin{proof}[Proof]
Here we will use the lower and the upper estimate given by thin-shell (\ref{eq-thin-shell-matrix}). 
Applying (\ref{eq-thin-shell-matrix}) for rank $1$ projections, we have:
$$
\mathbb{P} \left(1-\frac{\varepsilon}{2}\leqslant\left<Bx,x\right> \leqslant 1+\frac{\varepsilon}{2}\right) 
\geqslant 1-C\exp(-c\varepsilon^3\sqrt{m}) \quad \forall x\in \mathbb{S}^{n-1}.
$$
Define $\delta$, $\phi$ and $\alpha$ as follows:
$$
\delta = 1-\varepsilon; \quad \phi =\frac{\varepsilon}{3}
 \quad \text{ and } \quad \alpha =\frac{1-\frac{\varepsilon}{2}}{1-\varepsilon}.
$$
Recall some notations 
$$
A_0= 0,\ A_1=B_1,\ A_2=A_1+B_1,\ ..,\ A_N= A_{N-1} +B_N 
= \sum_{i=1}^N B_i.
$$ Denote 
$$
l_0= -\frac{n}{\phi},\ l_1=l_0+\delta,\ l_2=l_1+\delta,\ ..,\ l_N=l_{N-1}+\delta = N(1-\varepsilon) - \frac{3n}{\varepsilon}.
$$ Define 
$$
\phi_{l_i}(A_i) = {\rm Tr}\left(A_i-l_i.I_n\right)^{-1},
$$ the corresponding 
potential function when $A_i \succeq l_i.I_n$. 
Note also that $\delta\leqslant \frac{1}{\phi}$.\\
Denote by $\Im_i$ the 
event 
$$
\Im_i:=``A_i\succeq l_i.I_n \quad \text{ and }\quad \phi_{l_i}(A_i)\leqslant \phi ".
$$

Clearly $\mathbb{P}\left(\Im_0\right)=1$. Suppose now that $\Im_i$ 
is satisfied, following what was done after Lemma~\ref{lem-lower-shift}, 
condition (\ref{eq-condition-lower-bound}) is 
sufficient for the occurrence of the event $\Im_{i+1}$  :
$$
\frac{1}{\delta} q_2(\delta,B_{i+1}) -q_1(\delta,B_{i+1}) \geqslant 1
$$

Denoting $\lambda_j$ the eigenvalues of $A_{i}$ and $v_j$ 
the corresponding eigenvectors, taking the probability with 
respect to $B_{i+1}$ one can write :

\begin{align*}
&\mathbb{P}_{B_{i+1}}\left\{\frac{1}{\delta} q_2(\delta,B_{i+1}) -q_1(\delta,B_{i+1})< 1\right\}
\leqslant \\
&\leqslant\mathbb{P}_{B_{i+1}}\left\{\frac{1}{\delta} q_2(\delta,B_{i+1}) <\alpha\right\}
+\mathbb{P}_{B_{i+1}}\left\{q_1(\delta,B_{i+1})>\alpha -1\right\}\\
&\leqslant  \displaystyle\mathbb{P}_{B_{i+1}}\left\{\sum_{j=1}^n \frac{\left<B_{i+1}v_j,v_j\right>}{(\lambda_j-l_{i+1})^2} <\alpha \delta 
\sum_{j=1}^n \frac{1}{(\lambda_j-l_{i+1})^2}\right\}
 +\mathbb{P}_{B_{i+1}}\left\{ \sum_{j=1}^n \frac{\left<B_{i+1}v_j,v_j\right>}{\lambda_j-l_{i+1}}>\alpha -1\right\}\\
&\leqslant  \displaystyle\mathbb{P}_{B_{i+1}}\left\{\exists j\leqslant n\mid \left<B_{i+1}v_j,v_j\right> < 1-\frac{\varepsilon}{2} \right\}
+\mathbb{P}_{B_{i+1}}\left\{ \exists j\leqslant n\mid \left<B_{i+1}v_j,v_j\right> >\frac{\alpha -1}{\phi}\right\}\\
&\leqslant 2Cn.\exp(-c\varepsilon^3\sqrt{m}).
\end{align*}
Keeping in mind that $B_i$ are independent, we have shown that $$\mathbb{P}\left(\Im_{i+1}\vert \Im_i\right) 
\geqslant 1-2Cn.\exp(-c\varepsilon^3\sqrt{m}).$$ Moreover, we have 

\begin{align*}
\mathbb{P}\displaystyle \left\{ \lambda_{\min}\left( \frac{1}{N}\sum_{i=1}^N B_i\right)
\geqslant 1-\varepsilon -\frac{3n}{\varepsilon N}\right\}&\geqslant \mathbb{P}\left( \Im_N \right)\\
 &\geqslant \mathbb{P}\left( \Im_N \vert \Im_{N-1}\right)
 \mathbb{P}\left( \Im_{N-1} \vert \Im_{N-2}\right)\cdot \cdot\cdot\mathbb{P}\left(\Im_0\right)\\
&\geqslant 1-2CNn.\exp(-c\varepsilon^3\sqrt{m}).
\end{align*}
Therefore, Theorem\,\ref{lower-bound-log-concave} follows by the choice of $m$.\\

\end{proof}

\begin{Rq}\label{Rq-lower-bound-log-concave}
Note that in the previous proof, we used the large deviation 
inequality alongside the small ball probability estimate 
given by thin-shell (\ref{eq-thin-shell-matrix}). 
If one uses the deviation inequality given by (\ref{eq-large-deviation-matrix}) 
alongside the small ball probability estimate given by (\ref{eq-small-ball-matrix}),
then by the same proof it can be proved that 
$$
 \lambda_{\min}\left(\frac{1}{N}\sum_{i=1}^N B_i\right) \geqslant  -c+C\frac{n}{N},
$$
with high probability and with similar condition on $m$. The advantage of using 
thin-shell is that we can get an estimate close to $1$.
\end{Rq}

Combining the two previous results, we will be able to 
obtain, with high probability, a similar result to Proposition~\ref{prop-covariance-log-concave} for 
log-concave matrices: 

\begin{theo}\label{th-covariance-log-concave}
Let $A$ be an $n\times m$ isotropic log-concave matrix. 
For any $\varepsilon \in (0,1)$, if $m\geqslant \frac{C}{\varepsilon^6}\left[
\log (\frac{Cn}{\varepsilon})\right]^2$, taking 
$$
N\geqslant \frac{96n}{\varepsilon^2}
$$
copies $A_1,...,A_N$ of $A$, then with probability 
$\geqslant 1-\exp(-c\varepsilon^3\sqrt{m})$ we have 
$$
\left\Vert\frac{1}{N}\sum_{i=1}^N A_iA_i^t-I_n\right\Vert\leqslant \varepsilon .
$$
\end{theo}

\begin{proof}[Proof]
Let $\varepsilon\in (0,1)$,  $N\geqslant \frac{6n}{\varepsilon^{2}}$ and suppose 
that $m$ satisfies the assumption of the theorem. 
Note that
\begin{align*}
\mathbb{P}\left\{ \left\Vert\frac{1}{N}\sum_{i=1}^N A_iA_i^t-I_n\right\Vert > 4\varepsilon\right\}
&\leqslant \mathbb{P}\left\{ \lambda_{\max}\left(\frac{1}{N}\sum_{i=1}^N A_iA_i^t\right) > 1+2\varepsilon\right\}\\
&\quad + \mathbb{P}\left\{ \lambda_{\min}\left(\frac{1}{N}\sum_{i=1}^N A_iA_i^t\right) < 1-2\varepsilon\right\}
\end{align*}
and therefore it is sufficient to apply Theorem~\ref{upper-bound-log-concave} 
and Theorem~\ref{lower-bound-log-concave}.
\end{proof}

\subsection{Concrete examples of isotropic log-concave matrices}\label{section-exemple-log-concave}\hfill\\
For $x\in\mathbb{R}^k$, 
we denote by $x^*$ the vector with components $\vert x_i\vert $
arranged in nonincreasing order. Let $f: \mathbb{R}^k \longrightarrow \mathbb{R}$, 
we say that $f$ is absolutely symmetric 
if $f(x)=f(x^*)$ for all $x\in\mathbb{R}^k$. (For example, $\Vert \cdot\Vert_p$ 
is absolutely symmetric).\\
Define $F$ a function on $\mathbb{M}_{n,m}$ by $F(A)= f\left(s_1(A),..,s_k(A)\right)$ 
 for $A\in\mathbb{M}_{n,m}$ and $k=\min( n,m)$. It was shown by Lewis 
  \cite{MR1363368} that $f$ is absolutely symmetric if and only if $F$ is unitary 
 invariant. Moreover, $f$ is convex if and only if $F$ is convex.\\  
 Let $A$ be an $n\times m$ random 
 matrix whose density with respect to Lebesgue 
 measure is given by $G(A)=\displaystyle 
 \exp\left(-f\left(s_1(A),..,s_k(A)\right)\right)$, 
 where $f$ is an absolutely symmetric convex function. By the remark above, $G$
  is log-concave. This covers the case of random matrices with
  density of the form $\exp\left( -\sum_i V\left(s_i(A)\right)\right)$, where $V$ is an 
  increasing convex function on $\mathbb{R}^+$. When $V(x)=x^2$, this would be
   the gaussian unitary ensemble GUE.\\
   
\begin{prop}
Let $A$ be an $n\times m$ random matrix whose density with respect to Lebesgue 
is given by
$$
G(A)=\exp\left(-f(s_1(A),...,s_k(A))\right),
$$
where $f$ is an absolutely symmetric convex function and $k=\min(n,m)$. Suppose that 
$\mathbb{E} \Vert A\Vert_{HS}^2 =n$, then 
$A$ is an isotropic log-concave matrix, and $\sqrt{\frac{n}{m}}A^t$ is an $m\times n$ isotropic 
log-concave matrix.
\end{prop}   
   
 \begin{proof}
Since $f$ is an absolutely symmetric convex function, then 
$G$ is log-concave as we have seen above. It remains to prove the isotropic 
condition. 

Let $(a_{i,j})$ be the entries of $A$. Fix $(i,j)$ and $(k,l)$ two different indices. 
 Note $D_j= diag(1,..,-1,..,1)$ the $m\times m$ diagonal matrix 
 where the $-1$ is on the $j^{th}$ term. 
 Let $E_{(i,k)}$ be the $n\times n$ matrix obtained by swapping 
 the $i^{th}$ and $k^{th}$ rows in the identity matrix. Note 
 also $F_{(j,l)}$ the $m\times m$ matrix obtained by swapping 
 the $j^{th}$ and $l^{th}$ rows in the identity matrix.\\
 
 It is easy to see that $\alpha:= AD_j$ change the $j^{th}$ column of $A$ to its opposite and 
 keep the rest unchanged. Note that $AD_j$ has the same singular 
 values as $A$.\\
 
 Similarly, $\beta:=E_{(i,k)}AF_{(j,l)}$ permute $a_{i,j}$ with $a_{k,l}$ and keep the other 
 terms unchanged. Note also that $E_{(i,k)}AF_{(j,l)}$ has the same singular 
 values as $A$.\\
 
 Finally note that these $\alpha$ and $\beta$ have a Jacobian equal to $1$, and 
 since $f$ is absolutely symmetric, these transformations, which preserve the singular 
 values, don't affect the density.\\
 If $j\neq l$, by a change of variables $M=AD_j$ the density is invariant and 
 we have
 $$
 \int a_{i,j}a_{k,l} G(A)dA = -\int a_{i,j}a_{k,l} G(A)dA .
 $$
 Doing the change of variables $M= D_iA$ when $i\neq k$, we can conclude that 
 $$
  \mathbb{E} a_{i,j} a_{k,l}=0 \quad \text{ if } (i,j)\neq (k,l) .
  $$
  Now by a change of variables $M= E_{(i,j)}AF_{(k,l)}$ 
  the density is invariant and we have 
  $$
  \int a_{i,j}^2 G(A)dA = \int a_{k,l}^2G(A)dA .
  $$
  This implies that 
  $$
  \int a_{i,j}^2 G(A)dA = \frac{1}{nm}\sum_{k\leqslant n, l\leqslant m} \int a_{k,l}^2G(A)dA 
  =\frac{1}{nm}\int \Vert A\Vert_{\rm HS}^2G(A)dA = \frac{1}{m} .
  $$
  
  As a conclusion, we have shown that 
  $$
 \mathbb{E} a_{i,j} a_{k,l} = \frac{1}{m}\delta_{(i,j),(k,l)},
 $$
 which means that $A$ is isotropic. 
\end{proof}  

\begin{Rq}
The condition $\mathbb{E} \Vert A\Vert_{HS}^2=n$ can be obtained 
by a good normalization of the function $f$. Indeed, suppose that 
  $$
  \frac{1}{n}\int \Vert A\Vert_{\rm HS}^2 G(A)dA = c.
  $$
  Define $\hat{f}(x) = f(\sqrt{c} x)- nm \log (\sqrt{c})$ and $\hat{G}(A) = \exp\left(- \hat{f}(s_1(A),..,s_{k}(A))\right)$.\\
  
  Note that $\hat{G}$ is a probability density. Indeed, by the change of variables $M= \sqrt{c}A$ we have
  
  \begin{align*}
  \int \hat{G}(A)dA &= \int \exp\left(- f(\sqrt{c}s_1(A),..,\sqrt{c}s_{k}(A))\right)(\sqrt{c})^{nm}dA \\
  &=\int \exp\left(- f(s_1(M),..,s_{k}(M))\right)dM =1 .
  \end{align*}
 Note also that $\hat{G}$ satisfies the isotropic condition. Indeed, by the same change of variables 
  we can write
  $$
  \frac{1}{n}\int \Vert A\Vert_{\rm HS}^2\hat{G}(A)dA= \frac{1}{cn}\int \Vert M\Vert_{\rm HS}^2G(M)dM = 1 .
  $$
\end{Rq}

Applying Theorem~\ref{upper-bound-log-concave} and Theorem~\ref{lower-bound-log-concave} 
for $A$ and $A^t$ we get:

\begin{prop}
Let $A$ be an $n\times m$ random matrix whose density with respect to Lebesgue 
is given by
$$
G(A)=\exp\left(-f(s_1(A),...,s_k(A))\right),
$$
where $f$ is an absolutely symmetric convex function, properly normalized as above and $k=\min(n,m)$.\\
Suppose that $m\geqslant \frac{C}{\varepsilon^6}\left[
\log (\frac{Cn}{\varepsilon})\right]^2$ and $n\geqslant \frac{C}{\varepsilon^6}\left[
\log (\frac{Cm}{\varepsilon})\right]^2$, taking 
$N=\frac{96\max(n,m)}{\varepsilon^2}$ then with probability $\geqslant 1-\exp(-c\varepsilon^3\sqrt{k})$ we have
 $$
1-\varepsilon\leqslant \lambda_{\min}\left(\frac{1}{N}\sum_{i=1}^N A_iA_i^t\right)\leqslant \lambda_{\max}\left(\frac{1}{N}\sum_{i=1}^NA_iA_i^t\right)\leqslant 1+\varepsilon
$$
and
$$
(1-\varepsilon)\frac{n}{m}\leqslant \lambda_{\min}\left(\frac{1}{N}\sum_{i=1}^N A_i^tA_i\right)\leqslant \lambda_{\max}\left(\frac{1}{N}\sum_{i=1}^NA_i^tA_i\right)\leqslant (1+\varepsilon)\frac{n}{m} .
$$
\end{prop}

\vskip 0.5cm

\textbf{Acknowledgement :} I am grateful to my PhD advisor Olivier Guédon for his
encouragement, his careful review of this manuscript and many fruitful discussions. I would also like to thank 
Djalil Chafai and Alain Pajor for some interesting discussions and Matthieu Fradelizi 
for his precious help. I would also like to thank the anonymous referees for their valuable comments.

\nocite{*}
\bibliographystyle{abbrv}
\bibliography{bibliography-covariance}

\end{document}